\documentclass[a4paper,twoside]{article}  
\usepackage{bm}

\usepackage{amssymb,amsmath,amsthm,dsfont,amsfonts,color,latexsym}   

\usepackage{hyperref}
\usepackage{mathrsfs} 

\renewcommand{\baselinestretch}{1.2}
\def\bq{\begin{equation}}
\def\eq{\end{equation}}
\def\ba{\begin{array}{ccc}}
\def\bal{\begin{array}{lll}}
\def\ea{\end{array}}

\def\({\left(}\def\){\right)}
\def\[{\left[}\def\]{\right]}

    \def \R   {\mathbb{R}}

    \def\intr {\int_{\R^3}}

    \def\intt {\int^t_0}
    \def\intrr {\int_{\R^6}}

    \def \N    {\mathbb{N}}

    \def \pt   {\partial}
    
    \def \Dt   {\frac{\rm d}{{\rm d}t}}
    
    \def \dt    {\partial_t}

    \def \dxa   {\partial^{\alpha}_x}

    \def \dvb   {\partial^{\beta}_v}

    \def\Tdx   {\nabla_x}
    \def\Tdv   {\nabla_v}

       \def\be{\begin{equation}}
       \def\ee{\end{equation}}
       \def\bma#1\ema{{\allowdisplaybreaks\begin{align}#1\end{align}}}
       \def\bmas#1\emas{{\allowdisplaybreaks\begin{align*}#1\end{align*}}}
       \def\bln#1\eln{{\allowdisplaybreaks\begin{aligned}#1\end{aligned}}}
       \def\nnm{\notag}
       \def\bgr#1\egr{\allowdisplaybreaks\begin{gather}#1\end{gather}}
       \def\bgrs#1\egrs{\allowdisplaybreaks\begin{gather*}#1\end{gather*}}

       \theoremstyle{plain}
       \newtheorem{lem}{\bf Lemma}[section]
       \newtheorem{thm}[lem]{\textbf{Theorem}}

       \newtheorem{remark}[lem]{\bf Remark}

\begin{document}


\title{Green's Function and Pointwise Behaviors of the  Vlasov-Poisson-Fokker-Planck System}

\author{   Mingying Zhong$^*$\\[2mm]
 \emph
    {\small\it  $^*$College of Mathematics and Information Sciences,
    Guangxi University, P.R.China.}\\
    {\small\it E-mail:\ zhongmingying@sina.com}\\[5mm]
    }
\date{ }

\pagestyle{myheadings}
\markboth{Vlasov-Poisson-Fokker-Planck System}%
{ M.-Y. Zhong}

 \maketitle

 \thispagestyle{empty}

\begin{abstract}\noindent{
The pointwise space-time behaviors of the Green's function and the global solution to the Vlasov-Poisson-Fokker-Planck (VPFP) system in spatial three dimension are studied in this paper. It is shown that  the Green's function  consists of  the  diffusion waves decaying exponentially in time but algebraically in space, and the singular kinetic waves which become smooth for all $(t,x,v)$ when $t>0.$ Furthermore, we establish the pointwise space-time  behaviors of the global solution to the nonlinear VPFP system when the initial data is not necessarily smooth in terms of the Green's function.
}


\end{abstract}

%
\tableofcontents

\section{Introduction}
\label{sect1}
\setcounter{equation}{0}

The Vlasov-Poisson-Fokker-Planck  (VPFP)  system can be used to model the time evolution of dilute charged particles governed by the electrostatic force coming from their (self-consistent)
Coulomb interaction \cite{Markowich}. The collision term in the kinetic equation is the Fokker-Planck operator that describes the Brownian force. In general, the scaled  Vlasov-Poisson-Fokker-Planck (VPFP) system for one species reads
 \bgr
 \dt F+ v\cdot\Tdx F + \Tdx \Phi\cdot\Tdv F= \Tdv\cdot(\Tdv F+vF),\label{VPFP1}\\
  \Delta_x\Phi=\intr Fdv-\bar\rho,\label{VPFP2}\\
    F(0,x,v)=F_0(x,v),  \label{VPFP3}
 \egr
where $F=F(t,x,v)$ is the distribution function of charged particles with $(t,x,v)\in \R_+\times\R^3_x\times\R^3_v$, and $\Phi(t,x)$ denotes the
electrostatic potential. The background density $\bar\rho$ is assumed to be  constant  $1$ in this paper.

There are many important progress made on the well-posedness  and asymptotic behaviors of  solutions to the initial value problem
or the initial boundary value problem of the VPFP system. We refer to \cite{VPFP-1,VPFP-2,VPFP-3,Hwang} for results on the existence of classical solutions and
to \cite{VPFP-4,VPFP-5,VPFP-6,VPFP-7} for the existence of weak solutions and their regularity. Concerning the long-time
behavior of the VPFP system, we refer to \cite{time-1,time-2,time-3,Hwang,Li3}. The diffusion limit of the weak solution to the VPFP system had been studied
extensively in the literature (cf. \cite{FL-1,FL-2,FL-3,FL-4,FL-5}). The spectrum structure and the optimal decay rate of the classical solution to the VPFP system were investigated in \cite{Li3}.

We also mention some works related to the study in this paper. The Green's function and pointwise space-time behavior of the Boltzmann equation were first studied in the pioneering works~\cite{Lin2,Liu1,Liu2,Liu3}. For the linear Fokker-Planck equation, the structure of the Green's function was investigated in \cite{Kolmogorov,Lin1,Tanski1,Tanski2}. Recently, Li, Yang and Zhong established the Green's function and pointwise space-time behavior of the unipolar Vlasov-Poisson-Boltzmann system in \cite{Li4}.

In this paper, we study the pointwise space-time behaviors of the Green's function  and the global solution to the VPFP system~\eqref{VPFP1}--\eqref{VPFP2}  based on the spectral analysis~\cite{Li3}.
Note that the VPFP system \eqref{VPFP1}--\eqref{VPFP2} has an equilibrium state $(F_*,\Phi_*)=(M(v),0)$ with the normalized Maxwellian $M(v)$ given by
$$
 M=M(v)=\frac1{(2\pi)^{3/2}}e^{-\frac{|v|^2}2},\quad v\in\R^3.
$$

Set
$$
 F=M+\sqrt{M}f.
$$
Then Cauchy problem of the VPFP system~\eqref{VPFP1}--\eqref{VPFP2} and \eqref{VPFP3} for $f$ can be rewritten as
 \bgr
 \dt f+ v\cdot\Tdx f- v\sqrt{M}\cdot\Tdx \Phi- Lf
=  H(f),\label{VPFP4}\\
\Delta_x\Phi=\intr f\sqrt{M}dv,\label{VPFP5}\\
f(0,x,v)=f_0(x,v)=:\frac{F_0-M}{\sqrt{M}} ,\label{VPFP6}
\egr
where
 \bma
 Lf&= \Delta_vf-\frac{|v|^2}4f+\frac32f,\\
 H(f)& =\frac12 (v\cdot\Tdx\Phi)f- \Tdx\Phi\cdot\Tdv f .
 \ema

 We denote $L^2(\R^3)$ be a Hilbert space of complex-value functions $f(v)$
on $\R^3$ with the inner product and the norm
$$
(f,g)=\intr f(v)\overline{g(v)}dv,\quad \|f\|=\(\intr |f(v)|^2dv\)^{1/2}.
$$

It is obvious that $L$ is a non-positive self-adjoint operator on $L^2(\R^3_v)$. More precisely,
its Dirichlet form is given by
\be (Lf,f)=\intr \left|\Tdv f+\frac v2 f\right|^2dv. \ee

Therefore, the nullspace of the operator $L$, denoted by $N_0$, is a subspace spanned by $\sqrt{M}$.
Let $P_{0}$ be the projection operators from $L^2(\R^3_v)$
to the subspace $N_0$ with
\be
 P_{0}f=(f,\sqrt M)\sqrt M,   \quad P_1=I-P_{0}. \label{Pdr}
 \ee
Moreover, noting that
$$ L\sqrt M=0, \quad L(v_k\sqrt M)=-v_k\sqrt M, \,\,\, k=1,2,3, $$
 we introduce the projections $P_m,P_2$ in the invariant subspace of $L$ by
\be P_mf=\sum^3_{k=1}(f,v_k\sqrt M)v_k\sqrt M,\quad P_2=P_0+P_m,\quad P_3=I-P_2.\label{P23}\ee

Corresponding to the linearized operator $L$, we define the following dissipation norm:
\be \|f\|^2_{\sigma}=\|\Tdv f\|^2+\|\langle v\rangle f\|^2,\quad \langle v\rangle=\sqrt{1+|v|^2}. \label{norm2}\ee

From \cite{Yu}, the linearized operator $L$ is non-positive and
 locally coercive in the sense that there is a constant $0<\mu<1$ such that
\be
 (Lf,f)\le -\|P_1f\|^2,\quad  (Lf,f)\leq - \mu\|P_1f\|^2_{\sigma} .\label{L_4}
 \ee

Since we only consider the pointwise behavior with respect to the space-time variable $(t,x)$, it's convenient to regard the Green's function $G(t,x)$ as an  operator on $L^2(\R^3_v)$ defined by
\be   \label{LVPFP}
 \left\{\bln
 &\dt G =B G ,\quad t>0, \\
 &G(0,x)=\delta(x)I_v,
 \eln\right.
 \ee
 where $I_v$ is the identity in $L^2(\R^3_v)$ and the operator $B$ is given by
\be
 B f=Lf-v\cdot\Tdx f +v \cdot\Tdx(\Delta_x)^{-1} P_0f. \label{B(x)}
\ee
Then, the solution for the initial value problem of the linearized VPFP  equation
 \bq
 \left\{\bln
 &\dt f=B f, \\
 &f(0,x,v)=f_{0}(x,v),   \label{VPFP}
 \eln\right.
 \eq
can be represented by
\be f(t,x)=G(t)\ast f_0=\intr G(t,x-y)f_0(y)dy, \ee
where $f_0(y)=f_0(y,v).$

For any $(t,x)$ and $f\in L^2(\R^3_v)$, we define the $L^2$ norm of $G(t,x)$ by
\be \|G(t,x)\|=\sup_{\|f\|_{L^2_v}=1}\|G(t,x)f\|_{L^2_v}, \label{norm}\ee
and define the norms of operators $T_1$: $L^2(\R^3_v)\to L^2(\R^3_v)$ and $T_2$: $L^2(\R^3_v)\to \mathbb{C}$ by
\be \|T_1\|=\sup_{\|f\|_{L^2_v}=1}\|T_1f\|_{L^2_v},\quad |T_2|=\sup_{\|f\|_{L^2_v}=1}|T_2f|. \label{norm1}\ee

\noindent\textbf{Notations:} \ \ Before state the main results in this paper, we list some notations. For any $\alpha=(\alpha_1,\alpha_2,\alpha_3)\in \mathbb{N}^3$ and $\beta=(\beta_1,\beta_2,\beta_3)\in \mathbb{N}^3$, denote
$$\dxa=\pt^{\alpha_1}_{x_1}\pt^{\alpha_2}_{x_2}\pt^{\alpha_3}_{x_3},\quad \dvb=\pt^{\beta_1}_{v_1}\pt^{\beta_2}_{v_2}\pt^{\beta_3}_{v_3}.$$
The Fourier transform of $f=f(x,v)$
is denoted by
$$\hat{f}(\xi,v)=\mathcal{F}f(\xi,v)=\frac1{(2\pi)^{3/2}}\intr e^{-i x\cdot\xi}f(x,v)dx,$$
where and throughout this paper we denote $i=\sqrt{-1}$.

For $\gamma\ge 0$, define
$$ \|f(x)\|_{L^\infty_{v,\gamma}}=\sup_{v\in \R^3}|f(x,v)|(1+|v|)^\gamma.
$$


Define a pseudo-differential operator $\chi_R(D)$:
\be \chi_R(D)f(x)=\frac1{(2\pi)^{3/2}}\intr e^{ix\cdot\xi}\chi_R(\xi)\hat{f}(\xi)d\xi, \label{ch_R}
\ee
where $\chi_R(\xi)$ is a smooth cut-off function satisfying
\be
\chi_R(\xi)=0,\,\,\, |\xi|\le R;\quad \chi_R(\xi)=1,\,\,\, |\xi|\ge 2R. \label{ch_R1}
\ee


First, we have the pointwise space-time behaviors of the Green's function for the linearized VPFP system \eqref{LVPFP}.

\begin{thm}\label{green1}
Let $G(t,x)$ be the Green's function for the VPFP equation defined by \eqref{LVPFP}.
Then, there exists a small constant $R>0$ such that the Green's function $G(t,x)$ can be decomposed into
$$ G(t,x)=G_L (t,x)+G_H(t,x),$$
where  $G_L(t,x)=[I-\chi_R(D)]G(t,x)$ is the low frequency part and $G_H(t,x)=\chi_R(D)G(t,x)$ is the high frequency part with $\chi_R(D)$ defined by \eqref{ch_R}. Moreover, the following estimates hold for $G_L(t,x)$ and $G_H(t,x)$:

(1) For any multi-index $\alpha\in \mathbb{N}^3$, 
the low frequency part $G_{L}(t,x)$ satisfies
\be  \label{in1}
\left\{\bln
\|\dxa P_0 G_L(t,x)\|&\le  Ce^{-\frac14t}(1+|x|^2)^{-\frac{4+|\alpha|}2}  ,
\\
\|\dxa P_m G_L(t,x)\|&\le Ce^{-\frac14t}(1+|x|^2)^{-\frac{2+|\alpha|}2},
\\
\|\dxa P_3G_L(t,x) \| &\le  Ce^{-\frac14t}(1+|x|^2)^{-\frac{3+|\alpha|}2},
\eln\right.
\ee
and in particular
\be  \label{in4}
\left\{\bln
\|\dxa P_0 G_L(t,x)P_1\|&\le  Ce^{-\frac14t}(1+|x|^2)^{-\frac{5+|\alpha|}2}  ,
\\
\|\dxa P_m G_L(t,x)P_1\|
&\le Ce^{-\frac14t}(1+|x|^2)^{-\frac{3+|\alpha|}2} ,
\\
\|\dxa P_3G_L(t,x)P_1 \| &\le  Ce^{-\frac14t}(1+|x|^2)^{-\frac{4+|\alpha|}2},
\eln\right.
\ee
where $C>0$ is a constant dependent of  $\alpha$, and the operators $P_0$, $P_1$, $P_m$ and $P_3$ are defined by \eqref{Pdr} and \eqref{P23}.

(2) For any integer $n>0$ and multi-index $\alpha\in \mathbb{N}^3$, there exist a  constant $0<\eta_0<1/4$ such that the high frequency part $G_{H}(t,x)$ satisfies
 \be
 \|\dxa ( G_H(t,x)-W_\alpha (t,x))\|\le Ce^{-\eta_0 t}(1+|x|^2)^{-n},\label{in4a}
\ee
where $C>0$ is a constant dependent of $n$ and $\alpha$, and $W_\alpha(t,x)$ is the high frequency singular kinetic wave constructed by
\be
W_\alpha(t,x)= \sum^{7+[3|\alpha|/2]}_{k=0}\chi_R(D)J_k(t,x) , \label{wa}\ee
where
\bma
J_0(t,x)&=e^{tA}\delta(x)I_v, \label{J_0e}
\\
J_k(t,x)&=\intt e^{(t-s)A}(2+\Tdx\Delta^{-1}_xP_0)J_{k-1}ds,\quad k\ge 1. \label{J_1e}
\ema
Here $A=L-2-(v\cdot\Tdx)$, $I_v$ is an identity operator in $L^2_v$ and the operator $e^{ tA}$ is defined by
\be e^{ tA}h(x,v)=G_1(t)\ast h(x,v)=\intr\intr G_1(t,x,v;y,u)h(y,u)dydu, \label{G1}\ee
where $G_1(t,x,v;y,u)$ is the Green's function of the operator $A$ defined by \eqref{G11}.
\end{thm}

Then, we have the pointwise space-time behavior of the global solution to the nonlinear VPFP system~\eqref{VPFP4}--\eqref{VPFP6} as follows.

\begin{thm} \label{thm1}There exists a small constant $\delta_0>0$ such that if the initial data $f_0$ satisfies
\be
\| f_0(x)\|_{L^{\infty}_{v,3 }}\le C\delta_0(1+|x|^2)^{-n},\quad n\ge 2,\label{initial}
\ee
then there exists a unique global solution $(f,\Phi)$ to the VPFP system~\eqref{VPFP4}--\eqref{VPFP6} satisfying
\bma
\| P_0 f(t,x)\|_{L^2_v}&\le C\delta_0e^{-\eta_0t}(1+|x|^2)^{-2},\label{t1}\\
\| P_m f(t,x)\|_{L^2_v}+|\Tdx\Phi(t,x)|&\le C\delta_0e^{-\eta_0t}(1+|x|^2)^{-1}, \\
\| P_3 f(t,x)\|_{L^2_{v}}&\le C\delta_0e^{-\eta_0t}(1+|x|^2)^{-\frac32},\\
\| f(t,x)\|_{L^{\infty}_{v,3 }}+ h(t)\|\Tdv f(t,x)\|_{L^{\infty}_{v,2 }}&\le C\delta_0e^{-\eta_0t}(1+|x|^2)^{-1},\label{t7}
\ema where $h(t)=\sqrt{t}/(1+\sqrt{t})$.
Furthermore, if the initial data satisfies $(f_0, \sqrt{M})=0$ and \eqref{initial} for $n\ge 5/2$, then
\bma
\| P_0 f(t,x)\|_{L^2_v}&\le C\delta_0e^{-\eta_0t}(1+|x|^2)^{-\frac52},\label{t1a}\\
\|P_m f(t,x)\|_{L^2_v}+|\Tdx\Phi(t,x)| &\le C\delta_0e^{-\eta_0t}(1+|x|^2)^{-\frac32}, \\
\| P_3 f(t,x)\|_{L^2_{v}}&\le C\delta_0e^{-\eta_0t}(1+|x|^2)^{-2},\\
\| f(t,x)\|_{L^{\infty}_{v,3 }}+h(t)\|\Tdv f(t,x)\|_{L^{\infty}_{v,2 }}&\le C\delta_0e^{-\eta_0t}(1+|x|^2)^{-\frac32}.\label{t7a}
\ema
\end{thm}

\begin{remark}
The Green's function $G_1(t,x,v;y,u)$ becomes smooth for all $(t,x,v)$ when $t>0$  immediately. In particular,
\bmas
&G_1(t,x,v;y,u) \to e^{-2t}\frac{1}{t^6}e^{-\frac{3|x-y-(v+u)t/2|^2}{t^3}-\frac{|v-u|^2}{4t}},\quad t\to 0;\\
&G_1(t,x,v;y,u) \to e^{-2t}\frac{1}{t^{3/2}}e^{-\frac{|x-y-(v+u)|^2}{4t}-\frac{|v|^2+|u|^2}4 },\quad t\to \infty.
\emas
That is, $e^{2t}G_1(t,x,v;y,u)$ behaves like the Kolmogorov function at short time and like the heat kernal at large time. Thus, the singular waves $J_k(t,x)g_0$ with $g_0\in L^2(\R^3_v)$  are also smooth for all $(t,x,v)$ when $t>0$.
\end{remark}

\begin{remark}
Note that we can represent $G$ in terms of $G_1$ as
$$ G(t)=G_1(t)+\intt G_1(t-s)\ast (2 G+v\cdot\Tdx\Delta^{-1}_x P_0 G)ds. $$
Thus, the Green's function $G(t,x)g_0$ with $g_0\in L^2(\R^3_v)$  is smooth for all $(t,x,v)$ when $t>0$. Moreover, by the above relation, Theorem \ref{green1} and Lemmas \ref{S_1}--\ref{green3a}, we can obtain that for any $g_0(x,v)$ satisfying
$$
\| g_0(x)\|_{L^{\infty}_{v,\gamma }}\le C (1+|x|^2)^{-n},\quad n\ge 2, \,\, \gamma\ge 3,
$$
$G$ and its velocity directive $\Tdv G$ satisfy the following estimate
\bmas
\| (G(t)\ast g_0)(x)\|_{L^\infty_{v,\gamma}}&\le C e^{-\eta_0t}(1+|x|^2)^{-1} ,\nnm\\
\|\Tdv (G(t)\ast g_0)(x)\|_{L^\infty_{v,\gamma}}&\le C(1+t^{-\frac12})e^{-\eta_0t}(1+|x|^2)^{-1} .
\emas
\end{remark}

We now outline the main idea and make some comments on the proof of the above theorems. The results in Theorem~\ref{green1} on the pointwise behavior of the Green's function to the VPFP system  is proved based on the spectral analysis \cite{Li3} and the ideas inspired by \cite{Li4,Lin1}.  Indeed, we first decompose the Green function  $G$  into the lower frequency part  $G_{L}$ and the high frequency part  $G_{H}$.
By the virtue of the spectrum analysis of the VPFP system, the lower frequency part $G_{L}$  decays at $e^{-Ct}$, and in particular decays at $e^{- Ct/2}(1+|x|^2)^{-k}$ in the region $|x|\le e^{C t/(2k)}$.
Furthermore, we apply  the Picard's iteration  as the VPB system \cite{Li4}  to  estimate $G_L$ outside the region $|x|\le e^{C t/(2k)}$. To be more precisely,  we apply the macro-micro decomposition to construct the approximate sequence $(\hat{I}_k,\hat{J}_k) $ of $(P_2\hat{G}_L,P_3\hat{G}_L) $ such that $\hat{I}_k(t,\xi)$ and $\hat{J}_k(t,\xi)$ are the solutions to the Navier-Stokes-Poisson system with damping and the microscopic VPFP system respectively, and $\hat{I}_k(t,\xi)$ and $\hat{J}_k(t,\xi)$  are smooth and compact supported functions in $\xi$ and satisfy for any $g_0\in L^2(\R^3_v)$ (refer to Lemma \ref{ij}),
\bmas
\|\pt^\alpha_\xi \hat{I}_k(t,\xi)g_0\|_\xi&\le Ce^{-C_1t}|\xi|^{2k-|\alpha|-1}\|g_0\|,\\
\|\pt^\alpha_\xi \hat{J}_k(t,\xi)g_0\| &\le Ce^{-C_1t}|\xi|^{2k-|\alpha|}\|g_0\|.
\emas
Moreover, we establish the energy estimates on the remaining part $\hat{V}_k(t,\xi)$ and their frequency derivatives to show that  for  any $g_0\in L^2(\R^3_v)$ and any $\delta\in (0,\delta_0)$ (refer to Lemma \ref{lr1}),
$$\|\pt^\alpha_\xi \hat{V}_k(t,\xi)g_0\|_\xi\le  C\delta^{-2|\alpha|}e^{2\delta t}|\xi|^{2k-|\alpha|}\|g_0\| . $$
This implies that the remaining part $V_k(t,x)$ decays at  $e^{2\delta t}(1+|\delta x|^2)^{-k}$.
Combining the  above decompositions and estimates together, we can obtain the pointwise space-time behaviors of the low frequency part $G_{L}(t,x)$ as listed in \eqref{in1} and \eqref{in4}.

What left is to deal with the high frequency part $G_H(t,x)$. Since the Fourier transform of $G_H$ is not $L^1$ integrable in frequency space,  $G_H$ can be decomposed into the singular part and the remaining smooth part. We apply the refined Picard's iteration  as Fokker-Planck equation \cite{Lin1} to construct the singular kinetic waves $J_k$ of $G_H$ where  $J_k(t,x)$ are the solutions to the Fokker-Planck equations with damping given by \eqref{J_0e}--\eqref{J_1e}. In particular, $J_k(t,x)g_0$ with $g_0\in L^2(\R^3_v)$  are smooth for all $(t,x,v)$ when $t>0$, and   $\hat{J}_{k}(t,\xi)$ satisfy (refer to Lemma \ref{W_1})
$$
\| \pt^{\alpha}_{\xi}\hat{J}_{k}(t,\xi)g_0\|_\xi \le C (1+ t^{-3j/2})t^ke^{-t}(1+|\xi|)^{-j}\|g_0\|,\quad \forall j\ge 0.
$$
By the energy estimate in frequency space, we can  show that the remaining part $\hat{R}_k(t,\xi) $ is smooth and satisfies for any $l\ge 0$ and $\delta\in (0,\delta_0)$ (refer to Lemma \ref{W_2}),
$$ \|\pt^\alpha_\xi \hat{R}_k(t,\xi)g_0\|_\xi\le C\delta^{-1-2|\alpha|}e^{2\delta t}(1+|\xi|)^{-4-l}\|g_0\|, $$
where $k\ge 6+3l/2$.
This implies that the remaining part $R_k(t,x)$  decays at  $e^{-\eta_0t}(1+|x|^2)^{-n}$ for any $n\ge 1$ (refer to Lemma \ref{l-high2}). Combining the  above decompositions and estimates together, we can obtain the pointwise estimates of the high frequency part $G_{H}(t,x)$ as listed in \eqref{in4a} and \eqref{wa}.

Finally, making use of the  estimates of the Green's function, one can establish the pointwise space-time behaviors of global solution to the nonlinear VPFP system as in Theorem~\ref{thm1}. Compare to the results of the VPB system in \cite{Li4}, we can obtain the pointwise estimates of the solution when the initial data is not smooth. The main difficulty is that the nonlinear term $H(f)$ contain the high order derivative term $\Tdv f$.
To overcome this difficulty, we introduce the smooth  function $G_1(t,x,v;y,u)$   to represent the solution $f$ as
$$f (t)=G_1(t)\ast f_0+\intt G_1(t-s)\ast (2f+v\cdot\Tdx \Phi\sqrt M +H(f))ds.$$
By using the explicit expression of $G_1(t,x,v;y,u)$ in \eqref{G11}, we can obtain the key estimates of $ G_1(t)\ast g_0 $ and $\Tdv (G_1(t)\ast g_0)$ as in Lemmas \ref{S_1} and \ref{green3a} for any non smooth function $g_0(x,v)$. Combining the  above relation, Lemmas \ref{S_1}--\ref{green3a}  and Theorem \ref{green1}, we can obtain the pointwise space-time estimates of the global solution as in Theorem~\ref{thm1} after a straight forward computation.

The rest parts of this paper are organized as follows.
In Section~\ref{sect2}, we present the results about the spectrum analysis of the linear operator related to the linearized VPFP system. In Section~\ref{sect3}, we  establish the  pointwise space-time estimates  of the Green's functions to the  linearized VPFP system.
In Section~\ref{sect4}, we prove the pointwise space-time estimates of the global solutions to the original nonlinear VPFP system by making use of the estimates of the Green's fucntion.

\section{Spectral analysis}
\label{sect2}
\setcounter{equation}{0}
In this section, we review the spectrum structure of
 the linearized VPFP  operator $B(\xi)$ defined by \eqref{B(xi)} in order to study the pointwise estimate of the Green's function.

Taking the Fourier transform in \eqref{LVPFP} with respect to $x$, we obtain
 \bq
 \left\{\bln
 &\dt \hat{G}=B(\xi)\hat{G},  \quad t>0, \\
 &\hat{G}(0,\xi)=1(\xi)I_v,   \label{VPFP7}
 \eln\right.
 \eq
where the operator $B(\xi)$ is defined for $\xi\neq0$ by
 \be
B(\xi) =L-i(v\cdot\xi)-\frac{i(v\cdot\xi)}{|\xi|^2}P_0.  \label{B(xi)}
 \ee

Introduce the weighted Hilbert space $L^2_\xi(\R^3_v)$ for $\xi\ne 0$ as
$$
 L^2_\xi(\R^3)=\{f\in L^2(\R^3_v)\,|\,\|f\|_\xi=\sqrt{(f,f)_\xi}<\infty\},
$$
equipped with the inner product
$$
 (f,g)_\xi=(f,g)+\frac1{|\xi|^2}(P_0 f,P_0 g).
$$

We have the following results about the spectrum structure and semigroup of the operator $B(\xi)$.

\begin{lem}[\cite{Li3}]\label{SG_1}
The operator $B(\xi)$ generates a strongly continuous contraction semigroup on
$L^2_\xi(\R^3)$, which satisfies
\bq
\|e^{tB(\xi)}f\|_\xi\le\|f\|_\xi, \quad \forall\ t>0,\,\,f\in L^2_\xi(\R^3_v).
\eq
\end{lem}

\begin{lem}[\cite{Li3}] \label{eigen_3}
 Let $\sigma(B(\xi))$ denote the spectrum set of the operator $B(\xi)$. We have

(1)
There exists a constant $r_0 >0$ such that  for  $|\xi|\le r_0$,
\be
\sigma(B(\xi))\subset\{\lambda\in\mathbb{C}\,|\, {\rm Re}\lambda<-1/2\}, \label{sg1z}
\ee
and there exists a constant $\beta_0=\beta_0(r_0)>0$ such that for $|\xi|>r_0$,
\be
\sigma(B(\xi))\subset\{\lambda\in\mathbb{C}\,|\, {\rm Re}\lambda<-\beta_0\}. \label{sg2z}
\ee

(2) The semigroup $e^{tB(\xi)}$ satisfies
\bq \|e^{tB(\xi)}f\|_\xi\le Ce^{-\beta_1t}\|f\|_\xi, \quad \forall f\in L^2_\xi(\R^3_v),\label{B_0}\eq
where $\beta_1=\min\{\beta_0,1/2\}$ is a constant.

\end{lem}

\section{Green's function}
\setcounter{equation}{0}
\label{sect3}

In this section, we establish the pointwise space-time estimates of the Green's function defined by \eqref{LVPFP}. First, based on the spectral analysis given in section 2, we divide the Green's function into the low frequency part and the high frequency part and show that the low frequency part decaying exponential in time. Then, we introduce the Picard's iteration as VPB system \cite{Li4} to construct the approximate sequences  for the low frequency part, 
and estimate the remaining part by the energy method in frequency space. Finally, we refine the Picard's iteration in Fokker-Planck equation \cite{Lin1} to construct the approximate sequences (singular kinetic waves) for the high frequency part which are smooth for $(t,x,v)$ when $t>0$, 
and establish the pointwise behavior of the remaining part by the energy method in frequency space. 

\subsection{Low frequency part }
In this subsection, we establish the pointwise estimate of the low frequency part of the Green's function  based on the spectral analysis given in section 2.
To begin with, we decompose the operator $G(t,x)$ into low-frequency part and high-frequency part:
\be \label{L-R}
\left\{\bln
&G(t,x)=G_L(t,x)+G_H(t,x),\\
&G_L(t,x)=\frac1{(2\pi)^{3/2}}\intr e^{ i x\cdot\xi +tB(\xi)}\chi_1(\xi)d\xi,
\\
&G_H(t,x)=\frac1{(2\pi)^{3/2}}\intr e^{ i x\cdot\xi +tB(\xi)}\chi_2(\xi)d\xi,
\eln\right.
\ee
where
$$ \chi_1(\xi)=1-\chi_{R}(\xi),\quad \chi_2(\xi)=\chi_{R}(\xi),$$
and $R=r_0/2$ with $r_0>0$  given by Lemma \ref{eigen_3}.

From \cite{Li3}, we have the following estimates of each part of  $\hat{G}(t,\xi)$ defined by \eqref{VPFP7}.

\begin{lem}\label{l-1}
For any $g_0\in L^2_{\xi}(\R^3_v)$, there exists a positive constant $C$   such that
\bma
\|\hat{G}(t,\xi)g_0\|_{\xi},\|\hat{G}_{H}(t,\xi)g_0\|_{\xi}  &\le  Ce^{-\beta_1t}\|g_0\|_{\xi} , \label{l1}\\
\|\hat{G}_{L}(t,\xi)g_0\|_{\xi} &\le  Ce^{-\frac12t}\|g_0\|_{\xi} , \label{l2}
\ema
where $\hat{G}_L(t,\xi)$ and  $\hat{G}_{H}(t,\xi)$  are the Fourier transforms of $G_L(t,x)$ and $G_{H}(t,x)$  defined by \eqref{L-R}, and $\beta_1>0$ is given by Lemma \ref{eigen_3}.
\end{lem}

In the following, we  introduce the Picard's iteration  as \cite{Li4} and make use of the  energy estimates to establish the pointwise space-time behaviors of the low frequency part $G_L(t,x)$.  
To this end, we  apply the macro-micro decomposition to construct the approximate sequences to $\hat{G}_L(t,\xi)$ 
with increasing regularity at $\xi=0$, and deal with the remaining term by the weighted energy estimate.
Note that $\hat{G}_L(t,\xi)$ satisfies
\bgr
\dt \hat{G}_L+i(v\cdot \xi) \hat{G}_L-L\hat{G}_L+i(v\cdot\xi)|\xi|^{-2}P_0\hat{G}_L=0, \label{G_L}\\
G_L(0,\xi)=\chi_1(\xi)I_v.\nnm
\egr

Since
$$L\sqrt{M}=0,\quad L(v_j\sqrt{M})=-v_j\sqrt{M},\,\,\,\ j=1,2,3,$$
it follows that $N_1={\rm span}\{\sqrt M,v_j\sqrt M,\ j=1,2,3\}$ is an invariant subspace of $L$ and hence
$$P_2LP_3=P_3LP_2=0,\quad P_2LP_2=LP_2,\quad P_3LP_3=LP_3,$$
where $P_2$, $P_3$ are defined by \eqref{P23}.

We apply the projections $P_j$, $j=2,3$ to \eqref{G_L} to get
\bmas
\dt P_2\hat{G}_L+iP_2(v\cdot\xi P_2\hat{G}_L)-LP_2\hat{G}_L+i(v\cdot\xi)|\xi|^{-2}P_0\hat{G}_L&=-iP_2(v\cdot\xi P_3\hat{G}_L),
\\
\dt P_3\hat{G}_L+iP_3(v\cdot\xi P_3\hat{G}_L)-LP_3\hat{G}_L&=-iP_3(v\cdot\xi P_2\hat{G}_L).
\emas

Based on the above decomposition, we can construct an approximate solution   $(\hat{I}_k,\hat{J}_k)\in (N_1,N_1^\bot)$ for $(P_2\hat{G}_L,P_3\hat{G}_L)$ as follows
\bgr
\dt \hat{I}_0+iP_2(v\cdot\xi\hat{I}_0 )-L\hat{I}_0-i(v\cdot\xi)|\xi|^{-2}P_0\hat{I}_0=0,\label{I_0}\\
\dt \hat{J}_0+iP_3(v\cdot\xi\hat{J}_0 )-L\hat{J}_0=-iP_3(v\cdot\xi)\hat{I}_0, \label{J_0}\\
\hat{I}_0(0,\xi)=\chi_1(\xi)P_2,\quad \hat{J}_0(0,\xi)=\chi_1(\xi)P_3,\label{J_0b}
\egr
and
\bgr
\dt \hat{I}_k+iP_2(v\cdot\xi\hat{I}_k )-L\hat{I}_k-i(v\cdot\xi)|\xi|^{-2}P_0\hat{I}_k=-iP_2(v\cdot\xi)\hat{J}_{k-1},\label{I_k}\\
\dt \hat{J}_k+iP_3(v\cdot\xi\hat{J}_k )-L\hat{J}_k=-iP_3(v\cdot\xi)\hat{I}_{k}, \label{J_k}\\
\hat{I}_k(0,\xi)=0,\quad \hat{J}_k(0,\xi)=0,\label{J_kb}
\egr
for $k\ge 1$.
Denote
\bmas
B_1(\xi)&=-iP_2(v\cdot\xi)P_2-i(v\cdot\xi)|\xi|^{-2}P_0+LP_2 , \\
B_2(\xi)&=LP_3-iP_3(v\cdot\xi)P_3.
\emas
Since the operators $B_1(\xi)$ and $B_2(\xi)$ are dissipate on $N_1$ and $N_1^\bot$, it follows that $B_1(\xi)$ and $B_2(\xi)$ generate contraction semigroups on $N_1$ and $N_1^\bot$, respectively. By \eqref{I_0}--\eqref{J_kb}, we have
\bma
\hat{I}_0(t,\xi)&=e^{tB_1(\xi)}\chi_1(\xi)P_2, \label{I_0a}
\\
\hat{J}_0(t,\xi)&=e^{tB_2(\xi)}\chi_1(\xi)P_3-\intt e^{(t-s)B_2(\xi)}iP_3(v\cdot\xi)\hat{I}_0 ds,\label{J_0a}
\ema
and
\bma
\hat{I}_k(t,\xi)&=-\intt e^{(t-s)B_1(\xi)}iP_2(v\cdot\xi)\hat{J}_{k-1} ds, \label{I_ka}
\\
\hat{J}_k(t,\xi)&=-\intt e^{(t-s)B_2(\xi)}iP_3(v\cdot\xi)\hat{I}_k ds,\quad k\ge 1.\label{J_ka}
\ema
Define the approximate solution and the remaining part as
\bma
U_k(t,x)&=\sum^k_{n=0}(I_n +J_n)(t,x),\quad Y_k(t,x)= \Delta^{-1}_x(U_k,\sqrt{M}),\\
V_k(t,x)&=G_L(t,x)-U_k(t,x) ,\quad Z_k(t,x)= \Delta^{-1}_x(V_k,\sqrt{M}). \label{w-r}
\ema
Then, it follows that $\hat U_k(t,\xi)$ and $\hat Y_k(t,\xi)$ satisfy
\bmas
\dt \hat{U}_k+ iv\cdot\xi \hat{U}_k-L\hat{U}_k-i(v\cdot\xi) \sqrt{M}\hat{Y}_k&=iP_2(v\cdot\xi \hat{J}_k),\\
-|\xi|^2\hat{Y}_k &=(\hat{U}_k , \sqrt{M}),\\
\hat{U}_k(0,\xi)&=\chi_1(\xi)I_v,
\emas
and $\hat V_k(t,\xi)$ and $\hat Z_k(t,\xi)$ satisfy
\bma
\dt \hat{V}_k+ iv\cdot\xi \hat{V}_k-L\hat{V}_k-i(v\cdot\xi) \sqrt{M}\hat{Z}_k&= -iP_2(v\cdot\xi \hat{J}_k), \label{r4}\\
-|\xi|^2\hat{Z}_k &=(\hat{V}_k , \sqrt{M}),\nnm\\
\hat{V}_k(0,\xi)&=0.\nnm
\ema

For any $\xi\ne 0$, we define a weighted norm of the operator $T: L^2(\R^3_v)\to L^2(\R^3_v)$ as
$$\|T\|_\xi=\sup_{\|f\|_{L^2_v}=1}\|Tf\|_\xi.$$
Then, we have the following estimates for the approximate sequence $(I_k,J_k)$.

\begin{lem}\label{ij}For any $k\ge 0$ and  $\alpha\in \mathbb{N}^3$, we have
\bma
\|\dxa P_0 I_k(t,x)\|&\le Ce^{-\frac14t}(1+|x|^2)^{-\frac{4+|\alpha|}2},\label{I_kb}\\
\|\dxa P_m I_k(t,x)\| &\le Ce^{-\frac14t}(1+|x|^2)^{-\frac{2+|\alpha|}2}, \\
\|\dxa J_k(t,x)\|&\le Ce^{-\frac14t}(1+|x|^2)^{-\frac{3+|\alpha|}2},\label{J_kc}
\ema
and
\bma
\|\dxa P_0 I_k(t,x)P_1\|&\le Ce^{-\frac14t}(1+|x|^2)^{-\frac{5+|\alpha|}2},\label{I_kc}\\
\|\dxa P_m I_k(t,x)P_1\| &\le Ce^{-\frac14t}(1+|x|^2)^{-\frac{3+|\alpha|}2}, \\
\|\dxa J_k(t,x)P_1\|&\le Ce^{-\frac14t}(1+|x|^2)^{-\frac{4+|\alpha|}2},\label{J_kd}
\ema
 where  $C>0$ is a constant dependent on  $\alpha$.  In particular, $\hat{I}_k(t,\xi)$ and $\hat{J}_k(t,\xi)$ are smooth and supported in $\{\xi\in \R^3\,|\,|\xi|\le 2R\}$ satisfying
\bma
\|\pt^\alpha_\xi\hat{I}_k(t,\xi)\|_{\xi}&\le C\sum^{|\alpha|}_{n=0}t^{k+n}|\xi|^{2k-|\alpha|-1}e^{-\frac12 t},\label{I_kd}\\
\|\pt^\alpha_\xi\hat{J}_k(t,\xi)\|&\le C\sum^{|\alpha|}_{n=0}t^{k+n} |\xi|^{2k-|\alpha|}e^{-\frac12 t},\label{J_ke}
\ema
for any $|\alpha|\ge 0$.
\end{lem}

\begin{proof}
Define $H_1(t,x)$ and $H_2(t,x)$ by their fourier transforms:
$$
\hat H_1(t,\xi)=e^{tB_1(\xi)}\chi_1(\xi)P_2,\quad \hat H_2(t,\xi)=e^{tB_2(\xi)}\chi_1(\xi)P_3.
$$
To estimate $\hat{H}_1(t,\xi)$, we consider the eigenvalue problem of $B_1(\xi)$ as follows
\be \(-iP_2(v\cdot\xi)P_2-i\frac{v\cdot\xi}{|\xi|^2}P_0+LP_2\)\psi=\lambda \psi,\quad \xi\in \R^3.\label{r5}\ee

By taking inner product between \eqref{r5} and $\{\sqrt{M},v\sqrt{M}\}$, we obtain
\bmas
&\lambda \hat{n}=-i(\hat{m}\cdot\xi), \\
&\lambda \hat{m}=-i\(\xi+\frac{\xi}{|\xi|^2}\)  \hat{n}-\hat{m},
\emas
where $\hat{n}=(\psi,\sqrt M)$ and $\hat{m}=(\psi,v\sqrt M).$

It is easy to verify that $ \lambda_{j}(|\xi|)$ and $\psi_{j}(\xi)$, $j=0,1,2,3$ are the eigenvalues and eigenfunctions of $B_1(\xi)$ for $\xi\in \R^3$ satisfying
\be \label{G_1c}
\left\{\bln
\lambda_{0}(|\xi|)&= -\frac12- \frac{i}{2}\sqrt{4|\xi|^2+3},\\
\lambda_{1}(|\xi|)&= -\frac12+ \frac{i}{2}\sqrt{4|\xi|^2+3},\\
\lambda_{j}(|\xi|)&=-1,\,\,\, j=2,3,  \\
\psi_{k}(\xi)&=|\xi|a_{k}(|\xi|) \sqrt{M}+b_{k}(|\xi|)|\xi|^{-1}(v\cdot\xi)\sqrt{M},\,\,\,  k=0,1,\\
\psi_j(\xi)&=(v\cdot Y^j)\sqrt{M} ,\,\,\, j=2,3,\\
\big(\psi_j(\xi)&,\overline{\psi_k(\xi)}\big)_\xi=\delta_{jk},\,\,\, j,k=0,1,2,3,
\eln\right.
\ee
where $a_k(s),b_k(s)$ are analytic functions of $s$ satisfying $a_k=-ib_k/\lambda_k$ and $b_k^2=\lambda_k^2/(\lambda^2_k-s^2-1)$ for $k=0,1$, and $Y^j$ are orthonormal vectors satisfying $Y^j\cdot \xi=0$, $j=2,3$.

By \eqref{r5} and \eqref{G_1c}, we obtain
\be
\hat H_1(t,\xi)=\sum^3_{j=0}e^{\lambda_j(|\xi|)t}\psi_j(\xi)\otimes \(\langle \psi_j(\xi)|+\frac1{|\xi|^2}\langle P_0\psi_j(\xi)|\)\chi_1(\xi) . \label{G_1b}
\ee
Here for any $f,g\in L^2(\R^3_v)$, the operator $f\otimes\langle g|$ on $L^2(\R^3_v)$ is defined by \cite{Liu1,Liu2}
$$ f\otimes\langle g|h=(h,\overline{g})f,\quad h\in L^2(\R^3_v). $$
It follows from \eqref{G_1b} and \eqref{G_1c} that
\bmas
 P_0\hat{H}_1=&\sum_{j=0,1}e^{\lambda_j(|\xi|)t}\chi_1(\xi)a_j(|\xi|)\sqrt{M}\otimes \langle (|\xi|^2+1)a_j(|\xi|)\sqrt{M}+b_j(|\xi|)(v\cdot\xi)\sqrt{M}|,\\
P_m\hat{H}_1=&\sum_{j=0,1}e^{\lambda_j(|\xi|)t}\chi_1(\xi)b_j(|\xi|)\frac{v\cdot\xi}{|\xi|^2}\sqrt{M}\otimes \langle (|\xi|^2+1)a_j(|\xi|)\sqrt{M}+b_j(|\xi|)(v\cdot\xi)\sqrt{M}|\\
&+\chi_1(\xi)e^{-t}\sum^3_{k=1}v_k\sqrt{M}\otimes \bigg\langle v_k\sqrt{M}-\frac{(v\cdot\xi)\xi_k}{|\xi|^2}\bigg|,
\emas
which lead to
\be\label{G_1d}
\left\{\bln
\|\dxa P_0H_1(t,x)\|&\le Ce^{-\frac{1}{3}t} (1+|x|^2)^{-k},\\
\|\dxa P_mH_1(t,x)\|&\le Ce^{-\frac{1}{3}t}(1+|x|^2)^{-1-\frac{|\alpha|}{2}},
\\
\|\dxa P_0H_1(t,x)P_1\|&\le Ce^{-\frac{1}{3}t}(1+|x|^2)^{-k},\\
\|\dxa P_mH_1(t,x)P_1\|&\le Ce^{-\frac{1}{3}t}(1+|x|^2)^{-\frac{3}{2}-\frac{|\alpha|}{2}},
\eln\right.
\ee
where $k\ge 0$ is any integer, and $C>0$ is a constant dependent of $k$ and $\alpha$.

By \eqref{norm2}--\eqref{L_4},  the operator
$$
B_2(\xi_1+i\xi_0)=LP_3-iP_3(v\cdot\xi_1)P_3+P_3(v\cdot\xi_0)P_3,\quad (\xi_1,\xi_0)\in \R^3\times \R^3
$$
is dissipate for $ \xi_1\in \R^3$ and $|\xi_0|\le 2\mu$, namely,
 $$-{\rm Re}(B_2(\xi_1+i\xi_0)f,f)\ge  ( \mu-\frac12|\xi_0|)\|f\|^2_{\sigma},\quad \forall f\in N^\bot_1.$$
 It follows that $B_2(\xi)$ with $\xi=\xi_1+i\xi_0\in \mathbb{C}^3 $ generates a contraction semigroup $e^{tB_2(\xi)} $ in $N_1^\bot$ for  $ \xi_1\in \R^3$ and $|\xi_0|\le 2\mu$, which satisfies
\be \|e^{tB_2(\xi)}f\|\le e^{-(1- \frac{|\xi_0|}{2\mu})t}\|f\|, \quad \forall\ t>0,\,\,f\in N^\bot_1, \ee
where we have used
$$-(Lf,f)\ge (1-\frac{|\xi_0|}{2\mu})\|f\|^2+\frac{|\xi_0|}{2}\|f\|^2_{\sigma},\quad \forall f\in N^\bot_1.$$
Moreover, the semigroup
$e^{tB_2(\xi)} $ is analytic in $\{\xi\in \mathbb{C}^3\,|\, |\textrm{Im}\xi|\le 2\mu\}$. Rewrite
\be  e^{tB_2(\xi)} \chi_1(\xi)= e^{tB_2(\xi)}\frac1{(1+|\xi|^2)^2} (1+|\xi|^2)^2\chi_1(\xi). \label{G_7}\ee
By Cauchy theorem, it holds for any $g_0\in N^\bot_1$ and $|\xi_0|\le \mu$ that
\bma
&\left\|\intr e^{i x\cdot\xi}e^{tB_2(\xi)}g_0\frac1{(1+|\xi|^2)^2} d\xi\right\|\nnm\\
=&\left\|\intr e^{ix\cdot(\xi+i\xi_0)}e^{[L-iP_3(v\cdot\xi)P_3+P_3(v\cdot\xi_0)P_3]t}g_0\frac1{(1+|\xi+i\xi_0|^2)^2} d(\xi+i\xi_0)\right\|\nnm\\
\le& e^{-x\cdot\xi_0}\intr \left\|e^{[L-iP_3(v\cdot\xi)P_3+P_3(v\cdot\xi_0)P_3]t}g_0\frac1{(1+|\xi+i\xi_0|^2)^2}\right\|d\xi\nnm\\
\le& e^{-|\xi_0| |x|}e^{-\frac12 t}\|g_0\|. \label{G_6}
\ema
Since $(1+|\xi|^2)^2\chi_1(\xi)$ is a smooth and compact supported function, we can prove that for any $k>0$, there exists $C>0$ such that
\be
\left|\intr e^{i x\cdot\xi}(1+|\xi|^2)^2\chi_1(\xi) d\xi\right|\le C(1+|x|^2)^{-k}.\label{G_5}
\ee
Thus, it follows from \eqref{G_7}--\eqref{G_5} that
\bma
\|\dxa H_2(t,x)\|&\le C\intr e^{- \frac12 t}e^{-\frac34\mu|x-y|}(1+|y|^2)^{-k}dy \nnm\\
&\le C e^{-\frac12 t}(1+|x|^2)^{-k},\quad \forall k\in \N. \label{G_2d}
\ema
Combining  \eqref{I_0a}--\eqref{J_ka}, \eqref{G_1d} and \eqref{G_2d}, we can prove \eqref{I_kb}--\eqref{J_kd}.

Next, we prove  \eqref{I_kd} and \eqref{J_ke} as follows. For any $g_0\in L^2(\R^3_v)$, define
$$\hat{\mathrm{I}}_k(t,\xi)=\hat{I}_k(t,\xi)g_0,\quad \hat{\mathrm{J}}_k(t,\xi)=\hat{J}_k(t,\xi)g_0.$$
By \eqref{G_1c}, \eqref{G_1b} and
$$
{\rm Re}(B_2(\xi)f,f) \le -(f,f),\quad \forall f\in N^\bot_0,\,\, \xi\in \R^3,
$$
we have
\bma
\|e^{tB_1(\xi)}f_0\|_\xi&\le Ce^{-\frac12t}\|f_0\|_\xi,  \quad \forall f_0\in N_1,\label{G_3a}\\
\|e^{tB_2(\xi)}f_1\|&\le e^{- t}\|f_1\|,  \quad \forall f_1\in N^\bot_1. \label{G_3}
\ema

Thus, it follows from \eqref{G_3a}, \eqref{G_3} and \eqref{I_0a}--\eqref{J_ka} that
\bmas
 \|\hat{\mathrm{I}}_0\|_{\xi}
&\le Ce^{-\frac12t}\| P_2g_0\|_\xi\le C |\xi|^{-1}e^{-\frac12t}\|g_0\|,
\\
 \|\hat{\mathrm{J}}_0\| &\le e^{- t} \| P_3g_0\|+\intt e^{-(t-s)}|\xi|\|\hat{\mathrm{I}}_0\|ds \le C e^{-\frac12 t}\|g_0\|,
\emas
and
\bmas
\|\hat{\mathrm{I}}_k\|_{\xi}
\le&  C\intt e^{-\frac12(t-s)}|\xi|\|\hat{\mathrm{J}}_{k-1}\|ds
\le C t^k|\xi|^{2k-1}e^{-\frac12t}\|g_0\|,
\\
 \|\hat{\mathrm{J}}_k\|
 \le& C\intt e^{ -(t-s)}|\xi|\|\hat{\mathrm{I}}_{k}\|ds \le Ct^k|\xi|^{2k}e^{-\frac12t}\|g_0\|,\quad k\ge 1.
\emas

Taking the derivative $\pt^\alpha_\xi$ to \eqref{I_0} with $|\alpha|\ge 1$, we have
\be
\dt\pt^\alpha_\xi \hat{\mathrm{I}}_0+B_1(\xi)\pt^\alpha_\xi \hat{\mathrm{I}}_0 = G_{0,\alpha}, \label{I_2b}
\ee
where
$$
G_{0,\alpha}=\sum_{\beta\le \alpha,|\beta|\ge 1} C^\beta_\alpha \(-iP_2\pt^\beta_\xi(v\cdot\xi)\pt^{\alpha-\beta}_\xi\hat{\mathrm{I}}_0+i\pt^\beta_\xi\(\frac{v\cdot\xi}{|\xi|^2}\)P_0\pt^{\alpha-\beta}_\xi \hat{\mathrm{I}}_0\).
$$
By \eqref{I_2b} and \eqref{G_3a}, we have
\bmas
 \|\pt^\alpha_\xi\hat{\mathrm{I}}_0\|_{\xi}
\le &Ce^{-\frac12 t}|\pt^{\alpha}_\xi\chi_1(\xi)|\|P_2g_0\|_\xi\\
&+C\intt e^{-\frac12(t-s)}\bigg(\frac1{|\xi|}\|\pt^{\alpha-1}_\xi\hat{\mathrm{I}}_{0}\| +\sum^{|\alpha|}_{n= 1}\frac1{|\xi|^{n+1}}|\pt^{\alpha-n}_\xi(\hat{\mathrm{I}}_{0},\sqrt{M})|\bigg)ds\\
\le &C_\alpha \sum^{|\alpha|}_{n=0}t^{n}|\xi|^{-|\alpha|-1} e^{-\frac12 t}\|g_0\| .
\emas

Taking the derivative $\pt^\alpha_\xi$ to \eqref{J_0} with $|\alpha|\ge 1$, we have
\be
\dt\pt^\alpha_\xi \hat{\mathrm{J}}_0+iP_3\pt^\alpha_\xi ( v\cdot\xi \hat{\mathrm{J}}_0)-L\pt^\alpha_\xi \hat{\mathrm{J}}_0  = i P_3\pt^\alpha_\xi(v\cdot\xi\hat{\mathrm{I}}_0). \label{I_1}
\ee
Taking the inner product between \eqref{I_1} and $\pt^\alpha_\xi \hat{\mathrm{J}}_0$ and using Cauchy inequality, we obtain
\bmas
&\Dt \|\pt^\alpha_\xi\hat{\mathrm{J}}_0\|^2+ \frac32\|\pt^\alpha_\xi\hat{\mathrm{J}}_0\|^2+\frac{\mu}{4} \| \pt^\alpha_\xi \hat{\mathrm{J}}_0\|^2_{\sigma}\\
\le& C|\alpha|^2  \|\pt^{\alpha-1}_\xi \hat{\mathrm{J}}_0\|^2+C|\xi|^2\| \pt^\alpha_\xi \hat{\mathrm{I}}_0\|^2+C|\alpha|^2  \| \pt^{\alpha-1}_\xi\hat{\mathrm{I}}_0\|^2,
\emas
which leads to
\bmas
\|\pt^\alpha_\xi\hat{\mathrm{J}}_0\|^2\le &Ce^{-\frac32 t}|\pt^{\alpha}_\xi\chi_1(\xi)|^2\|P_3g_0\|^2\\
&+C\intt e^{-\frac32(t-s)}\(|\alpha|^2  \|\pt^{\alpha-1}_\xi \hat{\mathrm{J}}_0\|^2 +|\xi|^2\| \pt^\alpha_\xi \hat{\mathrm{I}}_0\|^2+|\alpha|^2  \| \pt^{\alpha-1}_\xi\hat{\mathrm{I}}_0\|^2\)ds\\
\le& C_\alpha \sum^{|\alpha|}_{n=0}t^{2n}|\xi|^{-2|\alpha|} e^{- t}\|g_0\|^2.
\emas
We take the derivative $\pt^\alpha_\xi$ to \eqref{I_k} with $|\alpha|\ge 1$ to get
\be
\dt\pt^\alpha_\xi \hat{\mathrm{I}}_k+B_1(\xi)\pt^\alpha_\xi \hat{\mathrm{I}}_k = G_{k,\alpha}, \label{I_2c}
\ee
where
\bmas
G_{k,\alpha}=&\sum_{\beta\le \alpha,|\beta|\ge 1} C^\beta_\alpha \(iP_2\pt^\beta_\xi(v\cdot\xi)\pt^{\alpha-\beta}_\xi\hat{\mathrm{I}}_k+i\pt^\beta_\xi\(\frac{v\cdot\xi}{|\xi|^2}\)P_0\pt^{\alpha-\beta}_\xi \hat{\mathrm{I}}_k\)\\
&+\sum_{\beta\le \alpha} C^\beta_\alpha iP_2\pt^\beta_\xi(v\cdot\xi )\pt^{\alpha-\beta}_\xi \hat{\mathrm{J}}_{k-1}.
\emas
It follows from \eqref{I_2c} and \eqref{G_3a} that
\bmas
 \|\pt^\alpha_\xi\hat{\mathrm{I}}_k\|_{\xi}
\le &C\intt e^{-\frac12(t-s)}\bigg(|\xi|\|\pt^\alpha_\xi\hat{\mathrm{J}}_{k-1}\|+\|\pt^{\alpha-1}_\xi\hat{\mathrm{J}}_{k-1}\|\\
&\qquad+\frac1{|\xi|}\|\pt^{\alpha-1}_\xi\hat{\mathrm{I}}_{k}\|  +\sum^{|\alpha|}_{n= 1}\frac1{|\xi|^{n+1}}|\pt^{\alpha-n}_\xi(\hat{\mathrm{I}}_{k},\sqrt{M})|\bigg)ds\\
\le &C_{\alpha}\sum^{|\alpha|}_{n=0}t^{k+n}|\xi|^{2k-|\alpha|-1} e^{-\frac12 t}\|g_0\|.
\emas

By taking the derivative $\pt^\alpha_\xi$ to \eqref{J_k} with $|\alpha|\ge 1$, we have
\be
\dt\pt^\alpha_\xi \hat{\mathrm{J}}_k+iP_1\pt^\alpha_\xi ( v\cdot\xi \hat{\mathrm{J}}_k)-L\pt^\alpha_\xi \hat{\mathrm{J}}_k  = i\pt^\alpha_\xi P_0(v\cdot\xi \hat{\mathrm{I}}_{k}). \label{I_3}
\ee
Taking the inner product between \eqref{I_3} and $\pt^\alpha_\xi \hat{\mathrm{J}}_k$ and using Cauchy inequality, we obtain
\bmas
& \Dt \|\pt^\alpha_\xi\hat{\mathrm{J}}_k\|^2+ \frac32\|\pt^\alpha_\xi\hat{\mathrm{J}}_k\|^2+\frac{\mu}{4} \| \pt^\alpha_\xi \hat{\mathrm{J}}_k\|^2_{\sigma}\\
\le&  C |\xi|^2\|\pt^\alpha_\xi  \hat{\mathrm{I}}_{k}\|^2 +C|\alpha|^2\|\pt^{\alpha-1}_\xi  \hat{\mathrm{I}}_{k}\|^2 +C |\alpha|^2 \|\pt^{\alpha-1}_\xi \hat{\mathrm{J}}_k\|^2,
\emas
which leads to
\bmas
\|\pt^\alpha_\xi\hat{\mathrm{J}}_k\|^2\le &C\intt e^{-\frac32(t-s)}\(|\alpha|^2  \|\pt^{\alpha-1}_\xi \hat{\mathrm{J}}_k\|^2 +|\xi|^2\| \pt^\alpha_\xi \hat{\mathrm{I}}_k\|^2+ |\alpha|^2 \| \pt^{\alpha-1}_\xi\hat{\mathrm{I}}_k\|^2\)ds\\
\le &C_{\alpha}\sum^{|\alpha|}_{n=0}t^{2(k+n)}|\xi|^{4k-2|\alpha|}e^{-t} \|g_0\|^2.
\emas
The proof is completed.
\end{proof}

With the help of Lemma \ref{ij}, we have the following the pointwise behaviors for the remaining terms $V_k(t,x)$ and $\Tdx Z_k(t,x)$ defined by \eqref{w-r}.

\begin{lem}\label{lr1} For any $k\ge 1$  and $\alpha\in \mathbb{N}^3$, there exists a small constant $\delta_0>0$ such that for any $\delta\in (0,\delta_0)$,
\be
\|\dxa V_k(t,x)\|+|\dxa \Tdx Z_k(t,x)|\le Ce^{2\delta t}(1+|\delta x|^2)^{-(k+\frac{3+|\alpha|}2)}, \label{Vk}
\ee
where $C>0$ is a constant dependent of $k$ and $\alpha$.
\end{lem}
\begin{proof}For any $g_0\in L^2(\R^3_v)$, define
$$\hat{\mathrm{V}}_k(t,\xi)=\hat{V}_k(t,\xi)g_0,\quad \hat{\mathrm{Z}}_k(t,\xi)=\hat{Z}_k(t,\xi)g_0.$$
We claim that for any $k\ge 1$ and $\alpha\in \mathbb{N}^3$,
\be \|\pt^\alpha_\xi \hat{\mathrm{V}}_k(t,\xi)\|_\xi\le C\delta^{-2|\alpha|}e^{2\delta t}|\xi|^{2k-|\alpha|}\|g_0\|. \label{v2}\ee

We prove \eqref{v2} by induction. Taking the inner product $(\cdot,\cdot)_{\xi}$ between $\hat{\mathrm{V}}_k $ and \eqref{r4} and choosing the real part, we have
\be
\frac12\Dt \|\hat{\mathrm{V}}_k\|^2_{\xi}-(L\hat{\mathrm{V}}_k,\hat{\mathrm{V}}_k)
={\rm Re}(iP_2(v\cdot\xi)\hat{\mathrm{J}}_k,\hat{\mathrm{V}}_k)_{\xi}. \label{n2}
\ee
Since
$(P_2(v\cdot\xi)\hat{\mathrm{J}}_k,\sqrt{M})=0,$
it follows from \eqref{n2} that
$$
\frac12\Dt \|\hat{\mathrm{V}}_k\|^2_{\xi}  +\frac{\mu}{2}\| P_1\hat{\mathrm{V}}_k\|^2_{\sigma}
\le  C|\xi|^2\|\hat{\mathrm{J}}_k\|^2,
$$
which leads to
\be
\|\hat{\mathrm{V}}_k\|^2_{\xi}\le C |\xi|^{4k+2}\|g_0\|^2 \intt s^{2k}e^{-s}ds
\le C |\xi|^{4k+2}\|g_0\|^2. \label{n4}
\ee

Suppose that \eqref{v2} holds for $|\alpha|\le j-1$. Taking the derivative $\pt^\alpha_\xi$ to \eqref{r4} with $|\alpha|=j$, we have
\be
\dt\pt^\alpha_\xi \hat{\mathrm{V}}_k+i v\cdot\xi \pt^\alpha_\xi\hat{\mathrm{V}}_k-L\pt^\alpha_\xi \hat{\mathrm{V}}_k+i\frac{v\cdot\xi}{|\xi|^2}P_0 \pt^\alpha_\xi\hat{\mathrm{V}}_k= G_{k,\alpha}, \label{I_2d}
\ee
where
$$
G_{k,\alpha}=\sum_{\beta\le \alpha,|\beta|\ge 1} C^\beta_\alpha \(i\pt^\beta_\xi(v\cdot\xi)\pt^{\alpha-\beta}_\xi\hat{\mathrm{V}}_k+i\pt^\beta_\xi\(\frac{v\cdot\xi}{|\xi|^2}\)P_0\pt^{\alpha-\beta}_\xi \hat{\mathrm{V}}_k\)+i P_2\pt^\alpha_\xi(v\cdot\xi \hat{\mathrm{J}}_{k}).
$$
Taking the inner product $(\cdot,\cdot)_\xi$ between \eqref{I_2d} and $\pt^\alpha_\xi \hat{\mathrm{V}}_k $ and using Cauchy inequality, we obtain
\bma
&\frac12\Dt \|\pt^\alpha_\xi \hat{\mathrm{V}}_k\|^2_{\xi}  +\frac{\mu}2\| P_1\pt^\alpha_\xi\hat{\mathrm{V}}_k\|^2_{\sigma} \nnm\\
\le&  \frac C{\delta} \(|\xi|^2\| \pt^\alpha_\xi\hat{\mathrm{J}}_{k}\|^2+\|\pt^{\alpha-1}_\xi  \hat{\mathrm{J}}_{k}\|^2\) +\frac C{\delta}  |\alpha|^2 \| P_1\pt^{\alpha-1}_\xi\hat{\mathrm{V}}_k\|^2\nnm\\
&+ \frac{C}{\delta}\bigg(\frac1{|\xi|^2}|\pt^{\alpha-1}_\xi\hat{m}_k|^2 +\sum^{|\alpha|}_{l=1}\frac1{|\xi|^{2l+2}}|\pt^{\alpha-l}_\xi \hat{n}_k|^2\bigg)+\delta\frac{1}{|\xi|^2} |\pt^\alpha_\xi \hat{n}_k|^2,\label{n6}
\ema
where $\delta\in (0,\delta_0)$ with $\delta_0>0$ sufficiently small, and
$$\hat{n}_k=(\hat{\mathrm{V}}_k,\sqrt{M}),\quad \hat{m}_k=(\hat{\mathrm{V}}_k,v\sqrt{M}).$$
Apply Gronwall's inequality to \eqref{n6} and using \eqref{v2}, we have
\bma
\|\pt^\alpha_\xi \hat{\mathrm{V}}_k\|^2_{\xi}\le&\frac{C}{\delta} |\xi|^{4k+2-2|\alpha|}\|g_0\|^2\delta^{2-2|\alpha|} \intt e^{2\delta(t-s)} e^{4\delta s}ds \nnm\\
&+\frac{C}{\delta} |\xi|^{4k+2-2|\alpha|}\|g_0\|^2\sum^{|\alpha|}_{n=0}\intt e^{2\delta(t-s)}s^{2k+2n}e^{- s}ds \nnm\\
\le &C\delta^{-2|\alpha|}|\xi|^{4k+2-2|\alpha|}e^{4\delta t}\|g_0\|^2.\label{n5}
\ema
This proves \eqref{v2} for $|\alpha|=j$.

Therefore, it holds for any $\gamma\in \mathbb{N}^3$,
\bmas
&\|\pt^\alpha_\xi(\xi^\gamma \hat{\mathrm{V}}_k)\|+|\pt^\alpha_\xi(\xi^\gamma\xi \hat{\mathrm{Z}}_k)|\\
\le&\sum_{\beta\le \alpha}C^\beta_\alpha\bigg(\Big\|\pt^\beta_\xi(\xi^\gamma)\pt^{\alpha-\beta}_\xi\hat{\mathrm{V}}_k\Big\| +\Big|\pt^\beta_\xi\(\frac{\xi^\gamma\xi}{|\xi|^2}\)\pt^{\alpha-\beta}_\xi\hat{n}_k\Big|\bigg)\\
\le&C\sum^{|\alpha|}_{|\beta|=0}|\xi|^{|\gamma|-|\beta|}\|\pt^{\alpha-\beta}_\xi\hat{\mathrm{V}}_k\|_\xi
\le C\delta^{-|\alpha|}|\xi|^{2k-|\alpha|+|\gamma|+1}e^{2\delta t}\|g_0\|,
\emas
which leads to
\bmas
&(\|\pt^\gamma_x \mathrm{V}_k(t,x)\|+|\pt^\gamma_x\Tdx \mathrm{Z}_k(t,x)|)x^{2\alpha} \\
\le& C \int_{|\xi|\le 2R} \(\|\pt^{2\alpha}_\xi (\xi^\gamma \hat{\mathrm{V}}_k)\|+|\pt^{2\alpha}_\xi(\xi^\gamma\xi \hat{\mathrm{Z}}_k)|\)d\xi\\
\le& Ce^{2\delta t}\delta^{-2|\alpha|}\|g_0\|\int_{|\xi|\le 2R}|\xi|^{2k-2|\alpha|+|\gamma|+1}d\xi\\
 \le& Ce^{2\delta t}\delta^{-2|\alpha|}\|g_0\|,
\emas for any  $|\alpha|\le k+|\gamma|/2+3/2$.
This proves the lemma.
\end{proof}

With the help of Lemmas \ref{ij} and \ref{lr1}, we can obtain the following pointwise space-time estimates of $G_L(t,x)$ defined by \eqref{L-R}.

\begin{thm}\label{low3}
For any $\alpha\in \mathbb{N}^3$, 
the low frequency part  $G_{L}(t,x)$  satisfies
\bma
\|\dxa P_0G_L(t,x)\|&\le Ce^{-\frac14t}(1+|x|^2)^{-\frac{4+|\alpha|}2} , \label{gl1}\\
\|\dxa P_mG_L(t,x)\|&\le Ce^{-\frac14t}(1+|x|^2)^{-\frac{2+|\alpha|}2} , \label{gl2}\\
\|\dxa P_3G_L(t,x)\|&\le Ce^{-\frac14t}(1+|x|^2)^{-\frac{3+|\alpha|}2} , \label{gl3}
\ema
where $C>0$ is a constant dependent of $\alpha$. Moreover,
\bma
\|\dxa P_0G_L(t,x)P_1\|&\le Ce^{-\frac14t}(1+|x|^2)^{-\frac{5+|\alpha|}2} , \label{gl4}\\
\|\dxa P_mG_L(t,x)P_1\|&\le Ce^{-\frac14t}(1+|x|^2)^{-\frac{3+|\alpha|}2} , \label{gl5}\\
\|\dxa P_3G_L(t,x)P_1\|&\le Ce^{-\frac14t}(1+|x|^2)^{-\frac{4+|\alpha|}2} . \label{gl6}
\ema
\end{thm}
\begin{proof}
By \eqref{w-r}, we can decompose $G_L(t,x)$ into
\be
G_L(t,x)=U_k(t,x)+V_k(t,x). \label{decom1}
\ee
By Lemma \ref{ij}, we have
\be
\left\{\bln
\|\dxa P_0U_k(t,x)\|&\le Ce^{-\frac14t}(1+|x|^2)^{-\frac{4+|\alpha|}2} ,\\
\|\dxa P_mU_k(t,x)\|&\le Ce^{-\frac14t}(1+|x|^2)^{-\frac{2+|\alpha|}2} , \\
\|\dxa P_3U_k(t,x)\|&\le Ce^{-\frac14t}(1+|x|^2)^{-\frac{3+|\alpha|}2} ,
\eln\right.
\ee
and
\be
\left\{\bln
\|\dxa P_0U_k(t,x)P_1\|&\le Ce^{-\frac14t}(1+|x|^2)^{-\frac{5+|\alpha|}2} ,\\
\|\dxa P_mU_k(t,x)P_1\|&\le Ce^{-\frac14t}(1+|x|^2)^{-\frac{3+|\alpha|}2} ,\\
\|\dxa P_3U_k(t,x)P_1\|&\le Ce^{-\frac14t}(1+|x|^2)^{-\frac{4+|\alpha|}2}.
\eln\right.
\ee
By \eqref{decom1}, Lemmas \ref{l-1} and \ref{ij}, we obtain
$$\|\hat{V}_k(t,\xi)\|_\xi\le  \|(\hat{G}_L -\hat{U}_k)(t,\xi)\|_\xi\le C|\xi|^{-1}e^{-\beta_2t},\quad |\xi|\le 2R,$$
where $\beta_2\in (0,1/2)$ is a constant, which leads to
$$
\|\dxa V_k(t,x)\|+|\dxa \Tdx Z_k(t,x)|\le Ce^{-\beta_2 t}.
$$
This together with \eqref{Vk} imply that
\bma
\|\dxa V_k(t,x)\|+|\dxa \Tdx Z_k(t,x)|&\le Ce^{-\frac{2\beta_2}{3} t+\frac{2\delta}3t}(1+|x|^2)^{-\frac{k}{3}-\frac{3+|\alpha|}6}\nnm\\
&\le Ce^{-\frac14 t}(1+|x|^2)^{-\frac{5+|\alpha|}2}, \label{Vk1}
\ema
for  $k\ge 6+|\alpha|$. Combining \eqref{decom1}--\eqref{Vk1}, we prove the theorem.
\end{proof}

\subsection{High frequency part }

In this subsection, we extract the singular part for the high frequency part $G_H$ and establish the pointwise estimate of the remaining part.
Since $\hat{G}_H $ does not belong to $L^1(\R^3_\xi)$,   $G_H$ can be decompose into the singular part and the remaining smooth part. Indeed, we apply a refined Picard's iteration as  \cite{Lin1} to construct the  approximate sequences of $\hat{G}_H $, and estimate the smooth remaining term by the energy estimate.
Note that $\hat{G}_H(t,\xi)$ satisfies
\bgrs
\dt \hat{G}_H+i(v\cdot \xi) \hat{G}_H-L\hat{G}_H+i(v\cdot\xi)|\xi|^{-2}P_0\hat{G}_H=0,\\
\hat{G}_H(0,\xi)=\chi_2(\xi)I_v.
\egrs

Set
\be A(\xi)=L-2-i(v\cdot\xi). \label{A2}\ee

\begin{lem} \label{mix1a}
For any $k\ge 0$ and $g_0\in L^2(\R^3_v)$, we have
\be
\|\Tdv e^{tA(\xi)}g_0\|^2\le C(1+t^{-1})e^{-4t}\|g_0\|^2, \label{mix2}
\ee
and
\be
\||\xi|^k e^{tA(\xi)}g_0\|^2\le C(1+t^{-3k})e^{-4t}\|g_0\|^2, \label{mix3}
\ee
for $C>0$ a positive constant.
\end{lem}
\begin{proof}
Let $h(t,x,v)$ be the solution of the following linear Fokker-Planck equation:
\bma
\dt h+v\cdot\Tdx h &=\Tdv\cdot(\Tdv h+vh), \label{fp1}\\
h(0,x,v)&=h_0(x,v). \label{fp2}
\ema
Then $h$ can be represented by
$$ h(t,x,v)=\intr\intr G_0(t,x,v;y,u)h_0(y,u)dydu, $$
where the Green's function $G_0$ is given by \cite{Tanski1}
\bmas
G_0(t,x,v;y,u)=&\frac1{(2\pi)^6}\(\frac{\pi}{\sqrt{ D(t)}}\)^3\exp\bigg\{-\frac1{4D(t)}\bigg[\frac12(1-e^{-2t})|\hat{x}|^2\nnm\\
&-\(2(1-e^{-t})-(1-e^{-2t})\)\hat{x}\cdot \hat{v}\nnm\\
&+\frac12\(2t-4(1-e^{-t})+(1-e^{-2t})\)|\hat{v}|^2\bigg]\bigg\}, 
\emas
and
\bmas
&D(t)=\frac12\(t(1-e^{-2t})-2(1-e^{-t})^2\),\\
&\hat x=x-(y+u(1-e^{-t})),\quad \hat v=v-ue^{-t}.
\emas
By a direct computation, we obtain
\bma
G_0(t,x,v;y,u)=&\frac1{(2\pi)^6}\(\frac{\pi}{\sqrt{ D(t)}}\)^3\exp\bigg\{ -\frac{1-e^{-2t}}{8D(t)}\left|x-y-(v+u)\frac{1-e^{-t}}{1+e^{-t}}\right|^2\nnm\\
&-\frac1{2(1-e^{-2t})} |v-ue^{-t}|^2 \bigg\}.
\ema

Set
$$h=e^{2t}\sqrt M g.$$
It follows from \eqref{fp1}--\eqref{fp2} that $g$ satisfies
\bmas
&\dt g+v\cdot\Tdx g-Lg+2g=0,\\
&g(0,x,v)=g_0(x,v)=M^{-1/2}h_0(x,v).
\emas
Thus
\be g(t,x,v)=e^{tA}g_0=\intr\intr G_1(t,x,v;y,u) g_0(y,u)dydu, \label{G21}\ee
where
\bma
G_1(t,x,v;y,u)=&e^{-2t}\frac{1}{\sqrt{ M(v)}}G_0(t,x,v;y,u)\sqrt{ M(u)}\nnm\\
=&\frac1{(2\pi)^6}\(\frac{\pi}{\sqrt{D(t)}}\)^3\exp\bigg\{ -\frac{1-e^{-2t}}{8D(t)}\left|x-y-(v+u)\frac{1-e^{-t}}{1+e^{-t}}\right|^2\nnm\\
&-\frac{1+e^{-2t}}{4(1-e^{-2t})} \left|v\frac{2e^{-t}}{1+e^{-2t}}-u\right|^2-\frac{1-e^{-2t}}{4(1+e^{-2t})}|v|^2-2t \bigg\}. \label{G11}
\ema
Taking the Fourier transform to \eqref{G21} with respect to $x$, we obtain
\be \hat{g}(t,\xi,v)=e^{tA(\xi)}\hat{g}_0= \intr \hat{G}_1(t,\xi,v;u) \hat{g}_0(\xi,u)du, \label{G31}\ee
where
\bma
\hat{G}_1(t,\xi,v;u)=&\frac1{(2\pi)^3}\(\frac{4}{1-e^{-2t}}\)^{3/2}\exp\bigg\{-i\xi\cdot(v+u)\frac{1-e^{-t}}{1+e^{-t}} -\frac{2D(t)}{1-e^{-2t}}|\xi|^2\nnm\\
&-\frac{1+e^{-2t}}{4(1-e^{-2t})} \left|v\frac{2e^{-t}}{1+e^{-2t}}-u\right|^2-\frac{1-e^{-2t}}{4(1+e^{-2t})}|v|^2-2t \bigg\}. \label{G13}
\ema
It is easy to verified that
\bma
\hat{G}_1(t,\xi,v;u)&=\hat{G}_1(t,\xi,u;v),\label{G22}\\
\intr |\hat{G}_1(t,\xi,v;u)|du&\le Ce^{-2t}e^{ -\frac{2D(t)}{1-e^{-2t}}|\xi|^2}. \label{G23}
\ema

Since $D(t)> 0,\,t>0$ is strictly increasing and satisfies
\be D(t)=\frac{1}{12}t^4+O(t^5),\,\,\, t\to 0;\quad D(t)=\frac12 t+O(1),\,\,\, t\to \infty,\label{Dt}\ee
it follows from \eqref{G31}, \eqref{G22}--\eqref{Dt} that for any $g_0\in L^2(\R^3_v)$,
\bmas
\||\xi|^k e^{tA(\xi)}g_0\|^2=&|\xi|^{2k}\intr\left|\intr  \hat{G}_1(t,\xi,v;u) g_0(u)du\right|^2dv\\
\le& |\xi|^{2k} \sup_v\intr | \hat{G}_1(t,\xi,v;u)|du\intr\(\intr | \hat{G}_1(t,\xi,v;u)|dv\) g^2_0(u)du \\
\le &C\(1+\frac1{t^{3k}}\)e^{-4t} e^{ -\frac{2D(t)}{1-e^{-2t}}|\xi|^2}\|g_0\|^2.
\emas

By a direct computation, one has
\bma
\left|\Tdv \hat{G}_1(t,\xi,v;u)\right|=&\bigg|\(-i\xi\frac{1-e^{-t}}{1+e^{-t}}-\frac{e^{-t}}{1-e^{-2t}}\(v\frac{2e^{-t}}{1+e^{-2t}}-u\)-\frac{1-e^{-2t}}{2(1+e^{-2t})}v\)\hat{G}_1(t,\xi,v;u)\bigg|\nnm\\
\le &C\bigg(\sqrt{\frac{(1-e^{-t})^{3}}{D(t)}}+ \frac{e^{-t}}{\sqrt{1-e^{-4t}}}+\sqrt{\frac{1-e^{-2t}}{1+e^{-2t}}}\bigg)e^{-2t}\(\frac{4}{1-e^{-2t}}\)^{3/2}\nnm\\
&\exp\bigg\{ -\frac{D(t)}{1-e^{-2t}}|\xi|^2-\frac{1+e^{-2t}}{8(1-e^{-2t})} \left|v\frac{2e^{-t}}{1+e^{-2t}}-u\right|^2-\frac{1-e^{-2t}}{8(1+e^{-2t})}|v|^2 \bigg\}. \label{G12}
\ema
Thus, it follows from \eqref{G31}, \eqref{Dt} and \eqref{G12} that for any $g_0\in L^2(\R^3_v)$,
\bmas
\|\Tdv e^{tA(\xi)}g_0\|^2=&\intr\left|\intr \Tdv  \hat{G}_1(t,\xi,v;u) g_0(u)du\right|^2dv\\
\le& \sup_v\intr |\Tdv  \hat{G}_1(t,\xi,v;u)|du\intr\(\intr |\Tdv  \hat{G}_1(t,\xi,v;u)|dv\) g^2_0(u)du \\
\le &C \(1+\frac1t\)e^{-4t} e^{ -\frac{4D(t)}{1-e^{-2t}}|\xi|^2}\|g_0\|^2.
\emas
This completes the proof.
\end{proof}

We define the approximate sequence $\hat{I}_k$ for the high frequency part $\hat{G}_H$ as follow
\be \label{c-w}
\left\{\bln
&\dt \hat{I}_0+i v\cdot \xi \hat{I}_0-(L-2)\hat{I}_0=0,\\
& |\xi|^2\hat{E}_0=-(\hat{I}_0, \sqrt{M}),\\
&\hat{I}_0(0,\xi)=\chi_2(\xi)I_v,
\eln\right.
\ee
and
\be \label{c-w1}
\left\{\bln
&\dt \hat{I}_k+i v\cdot \xi \hat{I}_k-(L-2)\hat{I}_k=2\hat{I}_{k-1}+i v\cdot\xi  \sqrt{M} \hat{E}_{k-1}, \\
&|\xi|^2\hat{E}_k=-(\hat{I}_k, \sqrt{M}), \\
&\hat{I}_k(0,\xi)=0,
\eln\right.
\ee
for $k\ge 1$.
Define the singular waves and the remaining part as
\bgr
W_k(t,x)=\sum_{j=0}^{k}I_j(t,x),\quad
\psi_k(t,x)=\sum_{j=0}^{k} E_j(t,x), \label{w_i}
\\
 R_k(t,x)=G_H(t,x)-W_k(t,x),\quad  \phi_k(t,x)=\Phi_H(t,x)-\psi_k(t,x) \label{W-R}
 \egr
wtih $\Phi_H(t,x)= \Delta^{-1}_x(G_H(t,x),\sqrt{M})$.

It follows from \eqref{c-w} and \eqref{c-w1} that $\hat{W}_k(t,\xi)$ and  $ \hat{\psi}_k(t,\xi) $ satisfy
\be \label{ew}
\left\{\bln
&\dt \hat{W}_k+i v\cdot \xi \hat{W}_k-L\hat{W}_k-i v\cdot \xi \sqrt{M} \hat{\psi}_k=-2\hat{I}_{k}-i v\cdot \xi \sqrt{M} \hat E_{k},\\
&|\xi|^2 \hat{\psi}_k=(\hat{W}_k, \sqrt{M}),\\
&\hat{W}_k(0,\xi)=\chi_2(\xi)I_v,
\eln\right.
\ee
and  $\hat{R}_k(t,\xi)$ and  $ \hat{\phi}_k (t,\xi)$ satisfy
\be \label{er}
\left\{\bln
&\dt \hat R_k+i v\cdot \xi \hat R_k-L\hat R_k-i v\cdot \xi \sqrt{M} \hat \phi_k=2\hat{I}_{k}+i v\cdot \xi \sqrt{M} \hat E_{k},\\
&|\xi|^2 \hat \phi_k=(\hat R_k, \sqrt{M}),\\
&\hat R_k(0,\xi)=0.
\eln\right.
\ee

With the help of Lemma \ref{mix1a},
we can show the spatial derivative estimate  of the approximate sequence $I_k(t,x)$ and $E_k(t,x)$ as follow.

\begin{lem}\label{W_1}
For each $k,j\ge 0$ and $\alpha\in \mathbb{N}^3$, 
$\hat{I}_{j}(t,\xi) $ and $\hat{E}_{j}(t,\xi) $  are smooth and supported in $\{\xi\in \R^3\,|\,|\xi|\ge R\}$, and they  satisfy the following estimates.

(1) For $0<t\le 1$,
\be
\||\xi|^k\pt^{\alpha}_{\xi}\hat{I}_{j}(t,\xi)\|+ ||\xi|^k\pt^{\alpha}_{\xi}(\xi \hat{E}_{j}(t,\xi))|\le C  t^{j-3k/2} ,\label{w1}
\ee

(2) For $t>1$,
\be
\||\xi|^k\pt^{\alpha}_{\xi}\hat{I}_{j}(t,\xi)\|+ ||\xi|^k\pt^{\alpha}_{\xi}(\xi \hat{E}_{j}(t,\xi))|\le C t^je^{-2 t} .\label{w2}
\ee
\end{lem}

\begin{proof}
First, we want to show that for any $0<t\le 1$,
\be \||\xi|^k\hat{I}_j(t)\|\le t^{-3k/2+j} ,\quad \forall j,k\ge 0. \label{ijk}\ee
The estimate of $\hat{I}_0$ is immediately from Lemma \ref{mix1a}.
For $j\ge 1$, it follows from \eqref{c-w1} that
$$I_j(t)=\intt e^{(t-s)A(\xi)}(2I_{j-1}+i(v\cdot\xi)|\xi|^{-2}P_0I_{j-1})ds.$$
By induction and Lemma \ref{mix1a}, we obtain
\bmas
|\xi|^k\|\hat{I}_j(t)\|\le& \int^{t/2}_{0}\||\xi|^ke^{(t-s)A(\xi)}\|(  \|\hat{I}_{j-1}\|+ |\xi|^{-1} \|\hat{I}_{j-1}\|)ds\\
&+\int^{t}_{t/2}\|e^{(t-s)A(\xi)}\|( |\xi|^k\|\hat{I}_{j-1}\|+ |\xi|^{k-1} \|\hat{I}_{j-1}\|)ds\\
\le& C\int^{t/2}_{0}(t-s)^{-3k/2}s^{j-1} ds+C\int^{t}_{t/2} s^{-3k/2+j-1} ds\\
\le& C t^{-3k/2+j} ,\quad \forall j\ge 1.
\emas
This proves \eqref{ijk}.

Next, we claim that for $0<t\le 1$ and $|\alpha|\ge 1$,
\be |\xi|^{k}\|\pt^{\alpha}_{\xi}\hat{I}_j(t)\|\le Ct^{-3k/2+j} ,\quad \forall j,k\ge 0.\label{ijk1} \ee
Taking $\pt^{\alpha}_{\xi}$ to \eqref{c-w} with $|\alpha|\ge 1$, we obtain
$$\dt \pt^{\alpha}_{\xi}\hat{I}_0+i(v\cdot\xi)\pt^{\alpha}_{\xi}\hat{I}_0-(L-2)\pt^{\alpha}_{\xi}\hat{I}_0=-\sum_{\beta\le \alpha, |\beta|=1}iv^{\beta}\pt^{\alpha-\beta}_{\xi}\hat{I}_0.$$
Then we can represent $\pt^{\alpha}_{\xi}\hat{I}_0$ as
$$\pt^{\alpha}_{\xi}\hat{I}_0=e^{tA(\xi)}\pt^{\alpha}_{\xi}\chi_2(\xi)-\sum_{\beta\le \alpha, |\beta|=1}\intt  e^{(t-s)A(\xi)}iv^{\beta}\pt^{\alpha-\beta}_{\xi}\hat{I}_0ds.$$

Since
\bmas
&\intr |\hat{G}_1(t,\xi,v;u)||u|^2du\\
=&C\(\frac{4}{1-e^{-2t}}\)^{3/2}e^{-2t}e^{-\frac{2D(t)}{1-e^{-2t}}|\xi|^2}
e^{-\frac{1-e^{-2t}}{4(1+e^{-2t})}|v|^2}\\
&\times \intr e^{-\frac{1+e^{-2t}}{4(1-e^{-2t})} \left|v\frac{2e^{-t}}{1+e^{-2t}}-u\right|^2}|u|^2du\nnm\\
\le &C e^{-2t}e^{-\frac{2D(t)}{1-e^{-2t}}|\xi|^2}e^{-\frac{1-e^{-2t}}{4(1+e^{-2t})}|v|^2}\bigg[\frac{4e^{-2t}}{(1+e^{-2t})^2}|v|^2+(1-e^{-t})\bigg]\\
\le &C\(1+\frac1t\) e^{-2t}e^{-\frac{2D(t)}{1-e^{-2t}}|\xi|^2},
\emas
it follows from \eqref{G22}--\eqref{G23} that for any $g_0\in L^2(\R^3_v)$,
\bma
&|\xi|^{2k}\|e^{tA(\xi)}|v|g_0\|^2\nnm\\
\le& |\xi|^{2k}\sup_v\intr |\hat{G}_1(t,\xi,v;u)||u|^2du\intr\(\intr |\hat{G}_1(t,\xi,v;u)|dv\) g_0^2(u) du\nnm\\
 \le &\(1+\frac1{t^{3k+1}}\) e^{-4t} e^{-\frac{2D(t)}{1-e^{-2t}}|\xi|^2}\|g_0\|^2. \label{G_4}
\ema

By induction, \eqref{ijk} and \eqref{G_4}, we obtain that for $0<t\le 1$ and $|\alpha|\ge 1$,
\bma
|\xi|^{k}\|\pt^{\alpha}_{\xi}\hat{I}_0\|\le &C|\xi|^{k}\| e^{tA(\xi)}\| +\int^{t/2}_0  \||\xi|^ke^{(t-s)A(\xi)}|v|\|\|\pt^{\alpha-1}_{\xi}\hat{I}_0\| ds\nnm\\
&+\int^t_{t/2}  \|e^{(t-s)A(\xi)}|v|\| |\xi|^k\|\pt^{\alpha-1}_{\xi}\hat{I}_0\| ds\nnm\\
\le &C\int^{t/2}_0(t-s)^{-\frac{3k+1}{2}}ds  +C \int^t_{t/2}(t-s)^{-\frac{1}{2}}s^{-\frac{3k}{2}}ds \nnm\\
\le &Ct^{-3k/2} ,\quad \forall k\ge 0. \label{ijk2}
\ema

Taking $\pt^{\alpha}_{\xi}$ to  \eqref{c-w1}  with $|\alpha|\ge 1$, we obtain
\bmas
&\dt \pt^{\alpha}_{\xi}\hat{I}_j+i(v\cdot\xi)\pt^{\alpha}_{\xi}\hat{I}_j-(L-2)\pt^{\alpha}_{\xi}\hat{I}_j\\
=& \pt^{\alpha}_{\xi}\(2\hat{I}_{j-1}+i(v\cdot\xi)|\xi|^{-2}P_0\hat{I}_{j-1}\)-\sum_{\beta\le \alpha, |\beta|=1}iv^{\beta}\pt^{\alpha-\beta}_{\xi}\hat{I}_j.
\emas
Then we can represent $\pt^{\alpha}_{\xi}\hat{I}_j$ for $j\ge 1$ as
\bmas
\pt^{\alpha}_{\xi}\hat{I}_j=&\intt  e^{(t-s)A(\xi)}\pt^{\alpha}_{\xi}\(2\hat{I}_{j-1}+i(v\cdot\xi)|\xi|^{-2}P_0\hat{I}_{j-1}\)ds\\
 &-\sum_{\beta\le \alpha, |\beta|=1}\intt  e^{(t-s)A(\xi)}iv^{\beta}\pt^{\alpha-\beta}_{\xi}\hat{I}_jds.
\emas
Thus, by induction, \eqref{ijk} and \eqref{G_4} we can obtain that for $|\alpha|\ge 1$ and $j\ge 1$,
\bma
|\xi|^{k}\|\pt^{\alpha}_{\xi}\hat{I}_j\|\le & C\int^{t/2}_0  \||\xi|^ke^{(t-s)A(\xi)}\|\bigg(\|\pt^{\alpha}_{\xi}\hat{I}_{j-1}\|+\sum_{\beta\le\alpha}\frac{1}{|\xi|^{1+|\beta|}}\|\pt^{\alpha-\beta}_{\xi} \hat{I}_{j-1}\|\bigg) ds\nnm\\
&+C\int^t_{t/2}  \|e^{(t-s)A(\xi)}\| |\xi|^k\bigg(\|\pt^{\alpha}_{\xi}\hat{I}_{j-1}\|+\sum_{\beta\le\alpha}\frac{1}{|\xi|^{1+|\beta|}}\|\pt^{\alpha-\beta}_{\xi} \hat{I}_{j-1}\|\bigg)ds \nnm\\
& +\int^{t/2}_0  \||\xi|^ke^{(t-s)A(\xi)}|v|\|\|\pt^{\alpha-1}_{\xi}\hat{I}_j\| ds\nnm\\
&+\int^t_{t/2}  \|e^{(t-s)A(\xi)}|v|\| |\xi|^k\|\pt^{\alpha-1}_{\xi}\hat{I}_j\| ds\nnm\\
\le &Ct^{-3k/2+j} ,\quad \forall k\ge 0. \label{ijk3}
\ema
Then,
\be
|\xi|^{k}|\pt^{\alpha}_{\xi}(\xi\hat{E}_j)|\le C\sum_{\beta\le \alpha}\pt^{\beta}_{\xi}\(\frac{\xi}{|\xi|^2}\)|\xi|^{k}|(\pt^{\alpha-\beta}_{\xi}\hat{I}_j,\sqrt M)|
\le Ct^{-3k/2+j} ,\label{ijk5}
\ee
for $0<t\le 1$. By combining \eqref{ijk2}--\eqref{ijk5}, we prove \eqref{ijk1}.

Next, we want to show that for $t>1$,
$$|\xi|^k\|\hat{I}_j(t)\|\le Ct^je^{-2 t} , \quad \forall j,k\ge 0.$$
Since
$$\hat{I}_0(t)=e^{(t-1)A(\xi)}\hat{I}_0(1),\quad t>1,$$
it follows from \eqref{ijk} and Lemma \ref{mix1a} that
\be |\xi|^k\|\hat{I}_0(t)\|\le Ce^{-2(t-1)}|\xi|^k\|\hat{I}_0(1)\|\le Ce^{-2 (t-1)} . \label{ijk6}\ee
Noting that
$$ \hat{I}_j(t)=e^{(t-1)A(\xi)}\hat{I}_{j}(1)+\int^t_1 e^{(t-s)A(\xi)}(2\hat{I}_{j-1}+i(v\cdot\xi)|\xi|^{-2}P_0\hat{I}_{j-1})ds,$$
we have
\bma
|\xi|^k\|\hat{I}_j(t)\|\le& Ce^{-2(t-1)}|\xi|^k\|\hat{I}_j(1)\|+\int^t_1e^{-2(t-s)}|\xi|^k(1+|\xi|^{-1})\|\hat{I}_{j-1}\|ds\nnm\\
\le& Ce^{-2(t-1)} +\int^t_1e^{-2(t-s)}s^{j-1}e^{-2s} ds\nnm\\
\le& Ct^je^{-2t} ,\quad t>1. \label{ijk7}
\ema
For $t>1$,  $\pt^{\alpha}_{\xi}\hat{I}_0 $ can be written as
$$\pt^{\alpha}_{\xi}\hat{I}_0(t)=e^{(t-1)A(\xi)}\pt^{\alpha}_{\xi}\hat{I}_0(1)-\sum_{\beta\le \alpha, |\beta|=1}\int^t_1  e^{(t-s)A(\xi)}iv^{\beta}\pt^{\alpha-\beta}_{\xi}\hat{I}_0ds.$$
By induction, \eqref{ijk6} and \eqref{G_4}, we obtain that for $t>1$ and $|\alpha|\ge 1$,
\bma
|\xi|^{k}\|\pt^{\alpha}_{\xi}\hat{I}_0\|\le &Ce^{-2(t-1)}|\xi|^{k}\|\pt^{\alpha}_{\xi}\hat{I}_0(1)\| +\int^{t}_1 \|e^{(t-s)A(\xi)}(|\xi|^k|v|\pt^{\alpha-1}_{\xi}\hat{I}_0)\| ds\nnm\\
\le &Ce^{-2(t-1)} +C \int^t_1\(1+(t-s)^{-\frac12}\)e^{-2(t-s)}s^{j-1}e^{-2s} ds \nnm\\
\le & Ct^je^{-2t} ,\quad \forall k\ge 0. \label{ijk8}
\ema
By the similar arguments as \eqref{ijk8} and \eqref{ijk5}, we can prove
$$|\xi|^k\|\pt^{\alpha}_{\xi}\hat{I}_j(t)\|+|\xi|^{k}|\pt^{\alpha}_{\xi}(\xi\hat{E}_j(t))|\le Ct^je^{-2 t} ,\quad j\ge 0,\,\,t>1.$$
The proof of the lemma is completed.
\end{proof}

With the help of Lemma \ref{W_1}, we can show the pointwise estimates of the remaining terms $R_k(t,x)$ and $\phi_k(t,x)$ defined by \eqref{W-R} as follows.

\begin{lem}\label{W_2} For any $n\ge 1$ and $\alpha\in \N^3$, there exists a small constant $\delta_0>0$  such that for any $\delta\in (0,\delta_0)$,
\be
\|\dxa R_k(t,x)\|+|\dxa \Tdx  \phi_k(t,x)|\le C\delta^{-1}e^{2\delta t}(1+|\delta x|^2)^{-n}, \label{x4}
\ee
where $k\ge 6+3|\alpha|/2$ and $C>0$ is a constant dependent of  $n$ and $\alpha$.
\end{lem}

\begin{proof}
For any $g_0\in L^2(\R^3_v)$, we define
$$\hat{\mathrm{R}}_k(t,\xi)=\hat{R}_k(t,\xi)g_0,\quad \hat{\mathrm{\Omega}}_k(t,\xi)=\hat{\phi}_k(t,\xi)g_0.$$
We claim that for any $l\ge 0$ and $\alpha \in \mathbb{N}^3$,
\be \|\pt^\alpha_\xi \hat{\mathrm{R}}_k(t,\xi)\|_\xi\le C\delta^{-1-2|\alpha|}e^{2\delta t}(1+|\xi|)^{-4-l}\|g_0\|, \label{v3}\ee
where $k\ge 6+3l/2$ and  $\delta\in (0,\delta_0)$.

We prove \eqref{v3} by induction. From Lemma \ref{W_1}, it holds for any $k\ge 6+3l/2$ and $\alpha\in \N^3$,
\be \|\pt^{\alpha}_{\xi}\hat{I}_k\|+|\pt^{\alpha}_{\xi}(\xi \hat{E}_k)|\le Ce^{-t}(1+|\xi|)^{-4-l}. \label{ie}\ee
Taking the inner product between \eqref{er} and $\hat{R}_k+|\xi|^{-2}P_0\hat{R}_k$ and using Cauchy inequality, we have
\be
\frac12\Dt \|\hat{\mathrm{R}}_k\|^2_{\xi}+\frac{\mu}2\| P_1\hat{\mathrm{R}}_k\|^2_{\sigma}
\le \frac{C}{\delta}(\|\hat{\mathrm{I}}_{k}\|^2_{\xi}+|\xi|^2|\hat{\mathrm{E}}_{k}|^2) +\delta\|P_0\hat{\mathrm{R}}_k\|^2_{\xi},\label{H_1}
\ee
where $\delta\in (0,\delta_0)$ with $\delta_0>0$ small enough, and
$$\hat{\mathrm{I}}_k(t,\xi)=\hat{I}_k(t,\xi)g_0,\quad \hat{\mathrm{E}}_k(t,\xi)=\hat{E}_k(t,\xi)g_0.$$
Applying Gronwall's inequality to \eqref{H_1} and using \eqref{w1}, we obtain
\bma
\|\hat{\mathrm{R}}_k\|^2_{\xi}
\le &\frac{C}{\delta}\|g_0\|^2\intt e^{2\delta(t-s)}(1+|\xi|)^{-8-2l}e^{-2s} ds\nnm\\
\le &C\delta^{-2}(1+|\xi|)^{-8-2l}e^{2\delta t}\|g_0\|^2. \label{r6a}
\ema

Suppose that \eqref{v3} holds for $|\alpha|\le j-1$. Taking the derivative $\pt^\alpha_\xi$ to \eqref{r4} with $|\alpha|=j$ to get
\be
\dt\pt^\alpha_\xi \hat{\mathrm{R}}_k+i v\cdot\xi \pt^\alpha_\xi\hat{\mathrm{R}}_k-L\pt^\alpha_\xi \hat{\mathrm{R}}_k-i\frac{v\cdot\xi}{|\xi|^2}P_0 \pt^\alpha_\xi\hat{\mathrm{R}}_k= G_{k,\alpha}, \label{I_2}
\ee
where
\bmas
G_{k,\alpha}=&\sum_{\beta\le \alpha,|\beta|\ge 1} C^\beta_\alpha \(i\pt^\beta_\xi(v\cdot\xi)\pt^{\alpha-\beta}_\xi\hat{\mathrm{R}}_k+i\pt^\beta_\xi\(\frac{v\cdot\xi}{|\xi|^2}\)P_0\pt^{\alpha-\beta}_\xi \hat{\mathrm{R}}_k\)\\
&+i \pt^\alpha_\xi(v\cdot\xi \hat{\mathrm{E}}_{k} )\sqrt{M}+2\pt^\alpha_\xi\hat{\mathrm{I}}_{k} .
\emas

Taking the inner product between \eqref{I_2} and $\pt^\alpha_\xi \hat{\mathrm{R}}_k+|\xi|^{-2}P_0\pt^\alpha_\xi \hat{\mathrm{R}}_k$ and using Cauchy inequality, we have
\bma
&\frac12\Dt \|\pt^\alpha_\xi \hat{\mathrm{R}}_k\|^2_{\xi}+\frac{\mu}2 \|\pt^\alpha_\xi P_1\hat{\mathrm{R}}_k\|^2_{\sigma}\nnm\\
\le&  \frac C{\delta} \(\|\pt^\alpha_\xi \hat{\mathrm{I}}_{k}\|^2_{\xi}+|\pt^\alpha_\xi (\xi \hat{\mathrm{E}}_{k})|^2\)+\frac{C}{\delta}\bigg( |\pt^{\alpha-1}_\xi\hat{m}_k|^2+\sum^{|\alpha|}_{|\beta|=1} |\pt^{\alpha-\beta}_\xi n_k|^2\bigg)\nnm\\
& +\frac C{\delta} \| \pt^{\alpha-1}_\xi P_1\hat{\mathrm{R}}_k\|^2+\delta \|P_0\pt^\alpha_\xi \hat{\mathrm{R}}_k\|^2_{\xi},\label{r7}
\ema
where
$$\hat{n}_k=(\hat{\mathrm{R}}_k, \sqrt{M}),\quad \hat{m}_k=(\hat{\mathrm{R}}_k,v \sqrt{M}).$$
Thus, it follows from \eqref{r6a} and \eqref{r7} that
\bmas
\|\pt^\alpha_\xi \hat{\mathrm{R}}_k\|^2_{\xi}\le& \frac{C}{\delta}\|g_0\|^2(1+|\xi|)^{-8-2l}\(\intt e^{2\delta(t-s)}e^{-2s}ds+\delta^{-2|\alpha|}\intt e^{2\delta(t-s)}e^{4\delta s}ds\)\nnm\\
\le& C\delta^{-2-2|\alpha|}(1+|\xi|)^{-8-2l} e^{4\delta t}\|g_0\|^2.
\emas
This proves \eqref{v3} for $|\alpha|=j$.

Therefore, it holds for any $\gamma\in \mathbb{N}^3$,
\bmas
&\|\pt^\alpha_\xi(\xi^\gamma \hat{\mathrm{R}}_k)\|+|\pt^\alpha_\xi(\xi^\gamma\xi \hat{\Omega}_k)|\\
\le&\sum_{\beta\le \alpha}C^\beta_\alpha\bigg(\Big\|\pt^\beta_\xi(\xi^\gamma)\pt^{\alpha-\beta}_\xi\hat{\mathrm{R}}_k\Big\| +\Big|\pt^\beta_\xi\(\frac{\xi^\gamma\xi}{|\xi|^2}\)\pt^{\alpha-\beta}_\xi\hat{n}_k\Big|\bigg)\\
\le&C\sum^{|\alpha|}_{|\beta|=0}|\xi|^{|\gamma|-|\beta|}\|\pt^{\alpha-\beta}_\xi\hat{\mathrm{R}}_k\|_\xi
\le C\delta^{-1-|\alpha|}(1+|\xi|)^{-4}e^{2\delta t}\|g_0\|,
\emas
where $k\ge 6+3|\gamma|/2,$ which leads to
\bmas
(\|\pt^{\gamma}_x\mathrm{R}_k(t,x)\|+|\pt^{\gamma}_x\Tdx \Omega_k(t,x)|)x^{2\alpha}
\le& C \int_{|\xi|\ge R} (\|\pt^{2\alpha}_\xi (\xi^\gamma\hat{\mathrm{R}}_k)\|+|\pt^{2\alpha}_\xi(\xi^\gamma\xi \hat{\Omega}_k)|)d\xi\\
 \le& C\delta^{-1-2|\alpha|}e^{2\delta t}\|g_0\|.
\emas
 This proves the lemma.
\end{proof}

\begin{thm}\label{l-high2}Let $G_H(t,x)$ be the high frequency part defined by \eqref{L-R}. 
For any  $n\ge0$ and $\alpha\in\mathbb{N}^3$, there exists a constant $0<\eta_0 <1/4$ such that
\be
\|\dxa [G_H(t,x)-W_k(t,x)]\|
\le Ce^{-\eta_0 t}(1+|x|^2)^{-n},\label{x5}
\ee
for $k\ge 6+3|\alpha|/2$ and $C>0$ a constant dependent of $n$ and $\alpha$, where $W_k(t,x)$ is the singular kinetic wave defined by \eqref{w_i}.
\end{thm}

\begin{proof}
By  \eqref{W-R}, we have
\be
G_H(t,x)=W_k(t,x)+R_k(t,x). \label{x1}
\ee
From \eqref{er}, \eqref{ie} and Lemma \ref{l-1}, we can obtain that for $k\ge 6+3|\alpha|/2$,
\bmas
\|\hat{R}_k(t,\xi)\|_{\xi}
&\le \intt \left\|\hat{G}_H(t-s)(2\hat{I}_{k}+i (v\cdot\xi) \sqrt{M}  \hat{E}_{k})(s)\right\|_\xi ds\\
&\le C\intt e^{-\beta_1(t-s)}(\|\hat{I}_{k}(s)\|_{\xi} +|\xi| |\hat{E}_{k}(s)|)ds\\
&\le Ce^{-\beta_1t}(1+|\xi|)^{-4-|\alpha|},
\emas
 which gives
\be
\|\dxa R_k(t,x)\|+|\dxa \Tdx\phi_k(t,x)|\le Ce^{-\beta_1  t}. \label{x3}
\ee
It follows from \eqref{x1},  \eqref{x3} and \eqref{x4} that
$$
\|\dxa [G_H(t,x)-W_k(t,x)]\|
\le Ce^{-\eta_0 t}(1+|x|^2)^{-n},
$$
where $\eta_0=\beta_1/2\in (0,1/4)$, which lead to \eqref{x5}.
\end{proof}

 Theorem \ref{green1} directly follows from Theorem \ref{l-high2} and Theorem \ref{low3}.

\section{The nonlinear system}
\setcounter{equation}{0}
\label{sect4}
 In this section, we prove Theorem \ref{thm1} on the pointwise behaviors of the global solution to the nonlinear VPFP  system with the help of the estimates of the Green's function given in Sections 3.

\begin{lem} \label{S_1}
Let $\gamma,n\ge 0$  and $\alpha\in \mathbb{N}^3$. If the function $g_0(x,v)$  satisfies
$$
 | g_0(x,v)|\le C(1+|x|^2)^{-n}(1+|v|)^{-\gamma},
$$
then we have
\bma
|G_1(t)\ast g_0(x,v)|&\le Ce^{-t}(1+|x|^2)^{-n}(1+|v|)^{-\gamma}, \label{wt1}\\
|\Tdv G_1(t)\ast g_0(x,v)|&\le Ct^{-\frac12}e^{-t}(1+|x|^2)^{-n}(1+|v|)^{-\gamma}, \label{wt2}
\ema
where $G_1$ is defined by \eqref{G11}. In addition,
\be
|W_{\alpha}(t)\ast g_0(x,v)|\le Ce^{-t}(1+|x|^2)^{-n}(1+|v|)^{-\gamma}, \label{wt3}
\ee
where  $W_{\alpha}$ is defined by \eqref{wa}.
\end{lem}

\begin{proof}First, we deal with \eqref{wt1} and \eqref{wt2}.
Since it follows from \eqref{Dt} that
\bmas
&\frac1{D(t)^{3/2}}\intr e^{-\frac{1-e^{-2t}}{8D(t)}\left|x-y-(v+u)\frac{1-e^{-t}}{1+e^{-t}}\right|^2 } (1+|y|^2)^{-n}dy\\
=&\frac1{D(t)^{3/2}}\(\int_{|y|\le y_0}+\int_{|y|\ge y_0}\)e^{-\frac{1-e^{-2t}}{8D(t)}\left|x-y-(v+u)\frac{1-e^{-t}}{1+e^{-t}}\right|^2 } (1+|y|^2)^{-n}dy\\
\le &C\frac{1}{(1-e^{-2t})^{3/2}}\bigg(e^{-\frac{1-e^{-2t}}{64D(t)}\left|x-(v+u)\frac{1-e^{-t}}{1+e^{-t}}\right|^2 }+\bigg(1+\left|x-(v+u)\frac{1-e^{-t}}{1+e^{-t}}\right|^2\bigg)^{-n}\bigg)\\
\le &C\frac{1}{(1-e^{-2t})^{3/2}}(1+t)^n\bigg(1+\left|x-(v+u)\frac{1-e^{-t}}{1+e^{-t}}\right|^2\bigg)^{-n},
\emas
where $y_0=\frac12\left|x-(v+u)\frac{1-e^{-t}}{1+e^{-t}}\right| $, it follows that
\bma
& \left|\intrr G_1(t,x,v;y,u)g_0(y,u)dydu\right|\nnm\\
\le& C e^{-2t}\frac1{D(t)^{3/2}}\intr\(\intr e^{-\frac{1-e^{-2t}}{8D(t)}\left|x-y-(v+u)\frac{1-e^{-t}}{1+e^{-t}}\right|^2 } (1+|y|^2)^{-n}dy\)\nnm\\
 &e^{-\frac{1-e^{-2t}}{4(1+e^{-2t})}|v|^2 }e^{-\frac{1+e^{-2t}}{4(1-e^{-2t})} \left|v\frac{2e^{-t}}{1+e^{-2t}}-u\right|^2 }(1+|u|)^{-\gamma}du\nnm\\
\le & C(1+t)^n\frac{e^{-2t}}{(1-e^{-2t})^{3/2}}e^{-\frac{1-e^{-2t}}{4(1+e^{-2t})}|v|^2 }\intr e^{-\frac{1+e^{-2t}}{4(1-e^{-2t})}\left|v\frac{2e^{-t}}{1+e^{-2t}}-u\right|^2 } \nnm\\
&\bigg(1+\left|x-(v+u)\frac{1-e^{-t}}{1+e^{-t}}\right|^2\bigg)^{-n}(1+|u|)^{-\gamma}du. \label{G_1}
\ema

By changing variable $v\frac{2e^{-t}}{1+e^{-2t}}-u\to z$, we obtain
\bma
 &\intr e^{-\frac{1+e^{-2t}}{4(1-e^{-2t})}\left|v\frac{2e^{-t}}{1+e^{-2t}}-u\right|^2 } \bigg(1+\left|x-(v+u)\frac{1-e^{-t}}{1+e^{-t}}\right|^2\bigg)^{-2n}du\nnm\\
=  &\intr e^{-\frac{1+e^{-2t}}{4(1-e^{-2t})} |z |^2 } \bigg(1+\left|x-v\frac{1-e^{-2t}}{1+e^{-2t}}+z\frac{1-e^{-t}}{1+e^{-t}}\right|^2\bigg)^{-2n}dz\nnm\\
= &\(\int_{|z|\le z_0}+\int_{|z|\ge z_0}\)e^{-\frac{1+e^{-2t}}{4(1-e^{-2t})} |z |^2 } \bigg(1+\left|x-v\frac{1-e^{-2t}}{1+e^{-2t}}+z\frac{1-e^{-t}}{1+e^{-t}}\right|^2\bigg)^{-2n}dz\nnm\\
\le &C\(\frac{1-e^{-2t}}{1+e^{-2t}}\)^{3/2} \bigg(1+\left|x-v\frac{1-e^{-2t}}{1+e^{-2t}}\right|^2\bigg)^{-2n}, \label{G_2}
\ema
where  $z_0=\frac{1+e^{-t}}{2(1-e^{-t})}\left|x-v\frac{1-e^{-2t}}{1+e^{-2t}}\right| $.
Moreover,
\bma
 &\intr e^{-\frac{1+e^{-2t}}{4(1-e^{-2t})}\left|v\frac{2e^{-t}}{1+e^{-2t}}-u\right|^2 } (1+|u|)^{-2\gamma}du\nnm\\
=  &\(\int_{|u|\le u_0}+\int_{|u|\ge u_0}\)e^{-\frac{1+e^{-2t}}{4(1-e^{-2t})}\left|v\frac{2e^{-t}}{1+e^{-2t}}-u\right|^2 } (1+|u|)^{-2\gamma}du\nnm\\
\le& C\(\frac{1-e^{-2t}}{1+e^{-2t}}\)^{3/2}\bigg(1+\left|v\frac{e^{-t}}{1+e^{-2t}}\right|\bigg)^{-2\gamma},\quad  u_0= \left|v\frac{e^{-t}}{1+e^{-2t}}\right|.\label{G_3d}
\ema
Thus, it follows from \eqref{G_1}--\eqref{G_3d} that
\bma
& \left|\intrr G_1(t,x,v;y,u)g_0(y,u)dudy\right|\nnm\\
\le &C(1+t)^n\frac{e^{-2t}}{(1-e^{-2t})^{3/2}}e^{-\frac{1-e^{-2t}}{4(1+e^{-2t})}|v|^2 }\(\intr e^{-\frac{1+e^{-2t}}{4(1-e^{-2t})}\left|v\frac{2e^{-t}}{1+e^{-2t}}-u\right|^2 } (1+|u|)^{-2\gamma}du\)^{1/2}\nnm\\
&\(\intr e^{-\frac{1+e^{-2t}}{4(1-e^{-2t})}\left|v\frac{2e^{-t}}{1+e^{-2t}}-u\right|^2 } \bigg(1+\left|x-(v+u)\frac{1-e^{-t}}{1+e^{-t}}\right|^2\bigg)^{-2n}du\)^{1/2}\nnm\\
\le &C(1+t)^ne^{-2t}\bigg(1+\left|x-v\frac{1-e^{-2t}}{1+e^{-2t}}\right|^2\bigg)^{-n}\bigg(1+\left|v\frac{e^{-t}}{1+e^{-2t}}\right|\bigg)^{-\gamma}e^{-\frac{1-e^{-2t}}{4(1+e^{-2t})}|v|^2 }. \label{G2}
\ema
Let
\bmas
I_1&=\bigg(1+\left|v\frac{e^{-t}}{1+e^{-2t}}\right|\bigg)^{-\gamma}e^{-\frac{1-e^{-2t}}{16(1+e^{-2t})}|v|^2 },\\
I_2&=\bigg(1+\left|x-v\frac{1-e^{-2t}}{1+e^{-2t}}\right|^2\bigg)^{-n}e^{-\frac{1-e^{-2t}}{16(1+e^{-2t})}|v|^2 }.
\emas
It holds that $I_1\le C(1+|v|)^{-\gamma} $ for $t<1$ and $I_1\le e^{- c|v|^2 } $ for $t\ge 1$, and $I_2\le C(1+|x|^2)^{-n}$ for $|x|\ge 2|v|\frac{1-e^{-2t}}{1+e^{-2t}}$ and
$I_2\le e^{-\frac{1+e^{-2t}}{64(1-e^{-2t})}|x|^2}\le C(1+|x|^2)^{-n}$ for $|x|\le 2|v|\frac{1-e^{-2t}}{1+e^{-2t}}$, namely,
\be I_1\le C(1+|v|)^{-\gamma},\quad I_2\le C(1+|x|^2)^{-n}. \label{G3}\ee
Thus, it follows from \eqref{G2}--\eqref{G3} that
\be
|G_1(t)\ast g_0(x,v)|\le C(1+t)^ne^{-2t}(1+|x|)^{-n}(1+|v|)^{-\gamma}e^{-\frac{1-e^{-2t}}{8(1+e^{-2t})}|v|^2 }. \label{G4}
\ee

Noting that
\bma
|\Tdv G_1(t,x,v;y,u)|=&\bigg|-\frac{1-e^{-2t}}{2(1+e^{-2t})}v+\(x-y-(v+u)\frac{1-e^{-t}}{1+e^{-t}}\)\frac{(1-e^{-t})^2}{4D(t)}\nnm\\
&-\frac{e^{-t}}{1-e^{-2t}}\(v\frac{2e^{-t}}{1+e^{-2t}}-u\)\bigg|G_1(t,x,v;y,u)\nnm\\
\le &C\bigg(\sqrt{\frac{1-e^{-2t}}{1+e^{-2t}}}+\sqrt{\frac{(1-e^{-t})^{3}}{D(t)}}+ \frac{e^{-t}}{\sqrt{1-e^{-4t}}}\bigg)e^{-2t}\nnm\\
&e^{-\frac{1-e^{-2t}}{16D(t)}\left|x-y-(v+u)\frac{1-e^{-t}}{1+e^{-t}}\right|^2 }e^{-\frac{1-e^{-2t}}{8(1+e^{-2t})}|v|^2-\frac{1+e^{-2t}}{8(1-e^{-2t})} \left|v\frac{2e^{-t}}{1+e^{-2t}}-u\right|^2 } , \label{G14}
\ema
we obtain by a similar argument as \eqref{G4} that
\be
|\Tdv  G_1(t)\ast g_0(x,v)|\le C(1+t^{-\frac12})(1+t)^{n}e^{-2t}(1+|x|)^{-n}(1+|v|)^{-\gamma}e^{-\frac{1-e^{-2t}}{16(1+e^{-2t})}|v|^2 }, \label{G5}
\ee
where we have used \eqref{Dt}.

Next, we prove \eqref{wt3} as follows. By \eqref{wa}, we have
$$W_{\alpha}(t)\ast g_0=\sum^{7+[3|\alpha|/2]}_{k=0}\chi_R(D)J_k(t,x,v),$$
where
\bmas
J_0(t,x,v)&= G_1(t)\ast g_0= \intr G_1(x-y,v,t;u)g_0(y,u)dudy, \\
J_k(t,x,v)&= \intt G_1(t-s)\ast (2+v\cdot\Tdx\Delta^{-1}_xP_0) J_{k-1}ds,\,\,\, k\ge 1.
\emas

To estimate $\chi_R(D) $, we decompose
\bma
\chi_R(D) &=\intr e^{i x\cdot\xi} d\xi+\intr e^{i x\cdot\xi} \chi_1(\xi)d\xi\nnm\\
&= \delta(x)+ F(x). \label{J_1c}
\ema
Since $\chi_1(\xi)$ is a smooth and compact supported function, it follows that for any  $k\ge 0$, 
$$|F(x)|\le C(1+|x|^2)^{-k}.$$
Combining   \eqref{G4} and \eqref{J_1c}, we obtain
\bma
|\chi_R(D)J_0(t,x,v)|\le& |J_0(t,x,v)|+|F(x)\ast J_0(t,x,v)|\nnm\\
\le& C(1+t)^ne^{-2t}(1+|v|)^{-\gamma}e^{-\frac{1-e^{-2t}}{8(1+e^{-2t})}|v|^2 }\nnm\\
&\bigg((1+|x|^2)^{-n} +\intr (1+|x-y|^2)^{-n-2} (1+|y|^2)^{-n}dy\bigg)\nnm\\
\le& Ce^{-t} (1+|x|^2)^{-n}(1+|v|)^{-\gamma}e^{-\frac{1-e^{-2t}}{8(1+e^{-2t})}|v|^2 }. \label{J_3}
\ema

Since
$$\left|\intr e^{i x\cdot\xi}\frac{\xi_j}{|\xi|^2}\chi_2(\xi)d\xi\right|\le  C\frac1{|x|^2}(1+|x|^2)^{-k},\quad j=1,2,3,$$
for any $k\ge 0$, we have
\bma
&\quad|\chi_R(D)\Tdx\Delta^{-1}_x(J_0(t,x),\sqrt{M})|\nnm\\
&\le C(1+t)^ne^{-2t} \(\int_{|x-y|\le \frac12}+\int_{|x-y|> \frac12}\) \frac{1}{|x-y|^2}(1+|x-y|^2)^{-n-2}(1+|y|^2)^{-n}dy\nnm \\
&=:C(1+t)^ne^{-2t}(I_3+I_4).\label{J_4}
\ema
Note that
\bma
I_3&\le C(1+|x|^2)^{-n}\int_{|x-y|\le \frac12} \frac{1}{|x-y|^2}(1+|x-y|^2)^{-n-2}dy \nnm\\
&\le C(1+|x|^2)^{-n}, \label{J_4a}\\
I_4&\le C\int_{|x-y|> \frac12} (1+|x-y|^2)^{-n-2}(1+|y|^2)^{-n}dy \nnm\\
&\le C(1+|x|^2)^{-n}. \label{J_4b}
\ema
Combining \eqref{J_4}--\eqref{J_4b}, we have
\be
|\chi_R(D)\Tdx\Delta^{-1}_x(J_0(t,x),\sqrt{M})|\le C(1+t)^ne^{-2t}(1+|x|^2)^{-n}. \label{J_4c}
\ee
Thus, it follows from \eqref{J_3}, \eqref{J_4c} and \eqref{G4} that
\bmas
|\chi_R(D)J_1(t,x,v)|=&   \left|\chi_R(D)\intt  G_1(s)\ast (2J_0+v\cdot\Tdx\Delta^{-1}_xP_0J_0)(t-s)ds\right|\\
\le &C e^{-t}\intt\intrr e^{s} G_1(s,x,v;y,u) (1+|y|)^{-n}(1+|u|)^{-\gamma}  dyduds \\
\le & Ce^{-t}\intt e^{-s}(1+s)^n(1+|x|)^{-n}(1+|v|)^{-\gamma} ds\\
\le &Ce^{-t} (1+|x|)^{-n}(1+|v|)^{-\gamma}.
\emas

By a similar argument as above, we obtain
$$
|\chi_R(D)J_k(t,x,v)|\le Ce^{-t} (1+|x|)^{-n}(1+|v|)^{-\gamma},\quad \forall k\ge 2,
$$
 which proves \eqref{wt3}. The proof of the lemma is completed.
\end{proof}

\begin{lem}\label{green3a}Let $\gamma,n\ge 0$. If the functions $F_i(t,x,v)$, $i=0,1,2$  satisfy
\bmas
 | F_k(t,x,v)|&\le Ct^{-\frac k2}e^{-bt}(1+|x|^2)^{-n}(1+|v|)^{-\gamma} ,\quad k=0,1,\\
 \| F_2(t,x)\|_{L^2_v}&\le Ce^{-bt}(1+|x|^2)^{-n},
\emas
with $0<b< 2,$ then we have
\bma
\bigg|\intt  G_1(t-s)\ast F_k(s,x,v)ds\bigg|&\le Ce^{-bt}(1+|x|^2)^{-n}(1+|v|)^{-\gamma-1}, \label{s4d}\\
\bigg|\intt \Tdv G_1(t-s)\ast F_k(s,x,v)ds\bigg|&\le C e^{-bt}(1+|x|^2)^{-n}(1+|v|)^{-\gamma}, \label{s4e}\\
\bigg|\intt  G_1(t-s)\ast F_2(s,x,v)ds\bigg|&\le Ce^{-bt}(1+|x|^2)^{-n}, \label{s4}\\
\bigg|\intt  W_{\alpha}(t-s)\ast F_k(s,x,v)ds\bigg|&\le Ce^{-bt}(1+|x|^2)^{-n}(1+|v|)^{-\gamma-1}, \label{s4f}
\ema
where $k=0,1$ and $C>0$ is a constant.
\end{lem}

\begin{proof} By \eqref{G4}, we have
\bmas
&\left|\intt  G_1(t-s)\ast F_k(s,x,v)ds\right|\\
\le& Ce^{-bt}\intt e^{-(2-b)s}(1+s)^n(1+|x|)^{-n}(1+|v|)^{-\gamma}e^{-\frac{1-e^{-2s}}{8(1+e^{-2s})}|v|^2 }(t-s)^{-\frac k2}ds\\
\le& Ce^{-bt}\intt e^{-(2-b)s}(1+s)^n(1+|x|)^{-n}(1+|v|)^{-\gamma-1}\frac{1}{\sqrt{1-e^{-2s}}}(t-s)^{-\frac k2}ds\\
\le& C e^{-bt}(1+|x|)^{-n}(1+|v|)^{-\gamma-1},\quad k=0,1.
\emas
By \eqref{G5}, we have
\bmas
&\left|\intt\Tdv  G_1(t-s)\ast F_k(s,x,v)ds\right|\\
\le& Ce^{-bt}\intt e^{-(2-b)s}(1+s^{-\frac12})(1+s)^{n}(1+|x|)^{-n}(1+|v|)^{-\gamma}(t-s)^{-\frac k2}ds\\
\le& Ce^{-bt}(1+|x|)^{-n}(1+|v|)^{-\gamma},\quad k=0,1.
\emas
Thus, we obtain \eqref{s4d} and \eqref{s4e}.

Finally, we want to show \eqref{s4}. By changing variable $v\frac{2e^{-t}}{1+e^{-2t}}-u\to z$, we obtain
\bma
&\intr  |G_1(t,x,v;y,u)|^2 du\nnm\\
= & C\frac{e^{-4t}}{D(t)^{3}}e^{-\frac{1-e^{-2t}}{2(1+e^{-2t})}|v|^2 }\intr  e^{-\frac{1-e^{-2t}}{4D(t)}\left|x-y-v\frac{1-e^{-2t}}{1+e^{-2t}}+z\frac{1-e^{-t}}{1+e^{-t}}\right|^2 }
 e^{-\frac{1+e^{-2t}}{2(1-e^{-2t})} |z|^2 } dz\nnm\\
\le &C \frac{e^{-4t}}{D(t)^{3}}e^{-\frac{1-e^{-2t}}{2(1+e^{-2t})}|v|^2 }\bigg\{ e^{-\frac{1-e^{-2t}}{16D(t)}\left|x-y-v\frac{1-e^{-2t}}{1+e^{-2t}} \right|^2 }\(\frac{1-e^{-2t}}{1+e^{-2t}}\)^{3/2}\nnm\\
&+e^{-\frac{(1+e^{-2t})(1+e^{-t})}{8(1-e^{-t})^3} \left|x-y-v\frac{1-e^{-2t}}{1+e^{-2t}}\right|^2}\(\frac{D(t)}{1-e^{-2t}}\)^{3/2}\(\frac{1+e^{-t}}{1-e^{-t}}\)^{3}\bigg\}. \label{G6}
\ema
By \eqref{Dt}, \eqref{G6} and a similar argument as \eqref{G4}, we obtain
\bmas
&\left|\intt  G_1(t-s)\ast F_2(s,x,v)ds\right|\\
\le &\intt\intr \|G_1(s,x,v;y)\|_{L^2_u} \|F_2(t-s,y)\|_{L^2_u}dyds\\
\le & Ce^{-bt}(1+|x|^2)^{-n}\intt e^{-(2-b)s}\bigg\{\frac{(1+s)^n}{(1-e^{-2s})^{\frac34}}+\[\frac{(1-e^{-s})^{3}}{D(s)}\]^{\frac34}\bigg\}ds\\
\le &C e^{-bt}(1+|x|^2)^{-n}.
\emas
Finally, by \eqref{s4d} and a similar argument as Lemma \ref{S_1}, we can prove \eqref{s4f}. 
\end{proof}

With the help of Theorem \ref{green1} and Lemmas \ref{S_1}--\ref{green3a}, we are able to prove Theorem \ref{thm1} as follows.

\begin{proof}[\underline{\textbf{Proof of Theorem \ref{thm1}}}]First, we deal with \eqref{t1}--\eqref{t7}.
Let $f$ be a solution to the IVP problem \eqref{VPFP4}--\eqref{VPFP6} for $t>0$. We can represent this solution as
\be
f(t,x,v)=G(t)\ast f_0+\intt G(t-s)\ast H(f)(s) ds,
\ee
where $H(f)$ is the nonlinear term defined by
$$ H(f)=\frac12 (v\cdot\Tdx\Phi)f-\Tdx\Phi\cdot\Tdv f.$$
Define
\bmas
 Q(t)=\sup_{0\le s\le t,\atop x\in \R^3 }
 &\Big\{
    \(\| P_0 f\|_{L^2_v}(1+|x|^2)^{\frac12}  +\| P_3 f\|_{L^2_v}\)e^{\eta_0s}(1+|x|^2)^{\frac{3}2}
   \nnm\\
&+(\| P_m f\|_{L^2_v} + |\Tdx\Phi|) e^{\eta_0s}(1+|x|^2)
\nnm\\
  & +\(\| f\|_{L^\infty_{v,3}}  + \|\Tdv f\|_{L^\infty_{v,2}} \sqrt{s}(1+\sqrt{s})^{-1} \)e^{\eta_0s}(1+|x|^2)  \Big\}.
 \emas
It holds that for  $0\le s\le t$,
\bma
| H(s,x,v)|&\le \frac12|\Tdx\Phi||vf|+|\Tdx\Phi||\Tdv f|\nnm\\
&\le CQ^2(t)(1+s^{-\frac12})e^{-2\eta_0s}(1+|x|^2)^{-2}(1+|v|)^{-2}.  \label{h1}
\ema

By Theorem \ref{green1}, we decompose
\bma
 G(t)\ast f_0&=  G_L(t)\ast f_0+W_0(t)\ast f_0+(G_H-W_0)(t)\ast f_0, \label{gt1}\\
\intt G(t-s)\ast H(s) ds&= \intt G_L(t-s)\ast H(s)ds+\intt W_{0}(t-s)\ast H(s)ds\nnm\\
&\quad+\intt (G_H-W_0)(t-s)\ast  H(s)ds. \label{gt2}
\ema
By \eqref{h1}--\eqref{gt2},  Theorem \ref{green1}  and Lemmas \ref{S_1}--\ref{green3a}, and noting that $(H(t,x),\sqrt{M})=0 $, we obtain
\bma
\| P_0 f(t,x)\|_{L^2_v}
&\le C\delta_0e^{-\eta_0t}(1+|x|^2)^{-2} +Ce^{-\eta_0t}(1+|x|^2)^{-2} Q^2(t), \label{density}\\
\| P_m f(t,x)\|_{L^2_v}&\le C\delta_0e^{-\eta_0t}(1+|x|^2)^{-1} +Ce^{-\eta_0t}(1+|x|^2)^{-1} Q^2(t), \label{momentum}\\
\| P_3 f(t,x)\|_{L^2_{v}}&\le C\delta_0e^{-\eta_0t}(1+|x|^2)^{-\frac{3}2}  +Ce^{-\eta_0t}(1+|x|^2)^{-\frac{3}2} Q^2(t). \label{micro}
\ema
In addition, it holds that
\bma
| \Tdx\Phi(t,x)|&= \left| \intr \frac{x-y}{|x-y|^3}(f(t,y),\sqrt{M})dy\right|\nnm\\
&\le Ce^{-\eta_0t}(1+|x|^2)^{-1} (\delta_0+Q^2(t)). \label{field}
\ema

By \eqref{VPFP4}, we have
$$
\dt f- Af=2f+v\cdot\Tdx \Phi \sqrt{M}+H,
$$
where $A=L-2-(v\cdot\Tdx) $. Thus, we can represent $ f$ as
\bq  f(t,x,v)= G_1(t)\ast f_0+\intt  G_1(t-s)\ast (2f+v\cdot\Tdx \Phi \sqrt{M}+H)ds.\label{s0}\eq
By Lemma \ref{S_1}, it holds that
\be
|G_1(t)\ast f_0(x,v)| \le C\delta_0e^{-t}(1+|x|^2)^{-2}(1+|v|)^{-3}.\label{s2}
\ee
By \eqref{density}--\eqref{micro}, we have
$$
\| f(t,x)\|_{L^2_{v}}\le  C(\delta_0+Q(t)^2)e^{-\eta_0t}(1+|x|^2)^{-1} ,
$$
which together with  \eqref{s4}, \eqref{h1} and \eqref{field}, we have
\bma
&\bigg|\intt  G_1(t-s)\ast(2f+v\cdot\Tdx \Phi \sqrt{M}+H)ds\bigg|\nnm\\
\le& C(\delta_0+Q(t)^2)e^{-\eta_0t}(1+|x|^2)^{-1}. \label{s2b}
\ema
Thus, it follows from \eqref{s0}--\eqref{s2b} that
$$
| f(t,x,v)|\le  C (\delta_0+Q(t)^2)e^{-\eta_0t}(1+|x|^2)^{-1}.
$$
By induction  and \eqref{s4d},  we have
\be
| f(t,x,v)|\le  C(\delta_0+Q(t)^2)e^{-\eta_0t}(1+|x|^2)^{-1}(1+|v|)^{-3}. \label{s1b}
\ee

Next, we estimate  $\Tdv f$ as follows.
By \eqref{s0}, we have
\bq  \Tdv f(t,x)=\Tdv G_1(t)\ast f_0+\intt \Tdv G_1(t-s)\ast (2f+v\cdot\Tdx \Phi \sqrt{M}+H)ds.\label{s1}\eq
By Lemma \ref{S_1}, it holds that
\be
|\Tdv G_1(t)\ast f_0(x,v)| \le C\delta_0t^{-\frac12}e^{-t}(1+|x|^2)^{-2}(1+|v|)^{-3}.\label{s3a}
\ee
By \eqref{s1b}, \eqref{field}, \eqref{h1} and \eqref{s4e}, we have
\bma
&\bigg|\intt \Tdv  G_1(t-s)\ast(2f+v\cdot\Tdx \Phi \sqrt{M}+H)ds\bigg|\nnm\\
\le& C(\delta_0+Q(t)^2)e^{-\eta_0t}(1+|x|^2)^{-1}(1+|v|)^{-2}. \label{s3}
\ema
Thus, it follows from \eqref{s1}--\eqref{s3} that
\be
|\Tdv  f(t,x,v)|\le  C (\delta_0+Q(t)^2)(1+t^{-\frac12})e^{-\eta_0t}(1+|x|^2)^{-1}(1+|v|)^{-2}.\label{v1}
\ee

Combining   \eqref{density}--\eqref{field}, \eqref{s1b} and \eqref{v1}, we have
$$
Q(t)\le C\delta_0+CQ(t)^2,
$$
from which \eqref{t1}--\eqref{t7} can be  verified so long as $\delta_0>0$ is small enough.
Similarly,
we can prove \eqref{t1a}--\eqref{t7a} for the case of $P_0f_0 =0$, the details are omitted.

Finally, we prove the existence of the solution $f$  by the following iteration: $f^0 \equiv0$,
\be
f^{n}(t,x,v)=G(t)\ast f_0+\intt G(t-s)\ast H(f^{n-1})ds,\quad n\ge 1. \label{fk}
\ee
We want to show that $\{f^n\}_n$ is a Cauchy sequence in $L^{\infty}(\R_+\times \R^3_x\times \R^3_v)$. By the above estimates and using induction argument, we can obtain that when $\delta_0>0$ sufficiently small,
\be
\| f^n(t,x)\|_{L^\infty_{v,3}}+h(t)\| \Tdv f^n(t,x)\|_{L^\infty_{v,2}}+|\Tdx \Phi^n(t,x)|\le  C \delta_0 e^{-\eta_0t}(1+|x|^2)^{-1}, \label{fk1}
\ee
where $h(t)=\sqrt{t}/(1+\sqrt{t})$, $\Phi^n(t,x)=\Delta^{-1}_x(f^n(t,x),\sqrt M)$, and $C>0$ is a constant independent of $n$.  Let
$$g^{n}=f^{n+1}-f^{n},\quad \phi^{n}=\Phi^{n+1}-\Phi^{n},\quad n\ge 0.$$
Then, $g^n$ satisfies
\be
g^n(t,x,v)= \intt G(t-s)\ast (H^n_1+H^n_2) ds,\quad n\ge 1, \label{gk}
\ee
where
$$
\left\{\bln
H^n_1&=\(\frac12vf^n+\Tdv f^n\)\cdot\Tdx\phi^{n-1},\\
H^n_2&=\(\frac12vg^{n-1} +\Tdv g^{n-1} \)\cdot\Tdx\Phi^{n-1}.
\eln\right.
$$
By \eqref{fk1}, \eqref{gk} and using induction argument, we have
$$
\| g^n(t,x)\|_{L^\infty_{v,3}}+h(t)\| \Tdv g^n(t,x)\|_{L^\infty_{v,2}}+|\Tdx \phi^n(t,x)|\le  (C \delta_0)^{n+1}e^{-\eta_0t}(1+|x|^2)^{-1} . $$
Thus $\{f^n\}_{n}$ is a Cauchy sequence in $L^{\infty}(\R_+\times \R^3_x\times \R^3_v)$,  there exists a unique function $f(t,x,v)$ such that
\be \|f^n(t,x)-f(t,x)\|_{L^\infty_{v,3}} \le (C \delta_0)^{n+1}e^{-\eta_0t}(1+|x|^2)^{-1} \to 0,\quad n\to \infty.\label{gk1}\ee
This together with \eqref{fk} and \eqref{gk1} imply that $f (t,x,v)$ satisfies the following equation:
$$
f(t,x,v)=G(t)\ast f_0+\intt G(t-s)\ast H(f) ds.
$$
Thus $f(t,x,v)$ is the unique solution to the VPFP system \eqref{VPFP4}--\eqref{VPFP6}.
This completes the proof.
\end{proof}

\bigskip
\noindent {\bf Acknowledgements:} The author would like to thank Haitao Wang for helpful suggestions.
The author is supported by National Natural Science Foundation of China Nos. 11671100 and 12171104, the National Science Fund for Excellent Young Scholars No. 11922107, and Guangxi Natural Science Foundation Nos. 2018GXNSFAA138210 and 2019JJG110010.



\begin{thebibliography}{99}
\setlength{\itemsep}{-4pt}
\renewcommand{\baselinestretch}{1}
\small

\bibitem{time-1} L. Bonilla, J.-A. Carrillo and J. Soler,  Asymptotic behavior of an initial-boundary
value problem for the Vlasov-Poisson-Fokker-Planck system. SIAM J. Appl. Math., 57(5), (1997), 1343-1372.

\bibitem{VPFP-1} F. Bouchut, Existence and uniqueness of a global smooth solution for the Vlasov-Poisson-Fokker-Planck system in three dimensions. J. Funct. Anal. 111(1), (1993), 239-258.

\bibitem{VPFP-4} F. Bouchut,  Smoothing effect for the non-linear Vlasov-Poisson-Fokker-Planck system.
J. Differ. Equ. 122(2), (1995), 225-238.

\bibitem{time-2} F. Bouchut, J. Dolbeault, On long asymptotics of the Vlasov-Fokker-Planck equation
and of the Vlasov-Poisson-Fokker-Planck system with coulombic and newtonian
potentials. Differ. Integral Equ. 8, (1995), 487-514.

\bibitem{VPFP-5} J.-A. Carrillo, Global weak solutions for the initial-boundary value problems to the
Vlasov-Poisson-Fokker-Planck system. Math. Methods Appl. Sci. 21, (1998), 907-938.

\bibitem{VPFP-6} J.-A. Carrillo and J. Soler,  On the initial value problem for the Vlasov-Poisson-Fokker-Planck system with initial data in $L^p$ spaces. Math. Methods Appl. Sci. 18(10), (1995), 825-839.

\bibitem{time-3} J.-A. Carrillo, J. Soler and J.-L. Vazquez, Asymptotic behaviour and self-similarity
for the three dimensional Vlasov-Poisson-Fokker-Planck system. J. Funct. Anal. 141, (1996), 99-132.


\bibitem{FL-1} N. El Ghani and N. Masmoudi, Diffusion limit of the Vlasov-Poisson-Fokker-Planck
system. Commun. Math. Sci. 8(2), (2010), 463-479.

\bibitem{FL-2} T. Goudon,  Hydrodynamic limit for the Vlasov-Poisson-Fokker-Planck system:
analysis of the two-dimensional case. Math. Models Methods Appl. Sci. 15(5), (2005), 737-752.

\bibitem{FL-3} T. Goudon, J. Nieto, F. Poupaud and J. Soler, Multidimensional high-field limit of the
electro-static Vlasov-Poisson-Fokker-Planck system. J. Differ. Equ. 213(2), (2005), 418-442.

\bibitem{Hwang} H.J. Hwang and J. Jang, On the Vlasov-Poisson-Fokker-Planck equation
near Maxwellian. Discrete Contin. Dyn. Syst. Ser. B, 18(3), (2013), 681-691.



\bibitem{Kolmogorov} A. Kolmogorov, Zufallige Bewgungen (zur Theorie der Brownschen Bewegung). Ann. of Math.(2) 35(1), (1934),
116-117.


\bibitem{Li3} H.-L. Li, T. Yang, J.-W. Sun and M.-Y. Zhong, Large time behavior of solutions to Vlasov-Poisson-Landau (Fokker-Planck) equations
(in Chinese). Sci. Sin. Math., 46 (2016), 981-1004.

\bibitem{Li4} H.-L. Li, T. Yang and M. Zhong,  Green's Function and Pointwise Space-time Behaviors of the Vlasov-Poisson-Boltzmann System, Arch. Rational Mech. Anal., 235 (2020), 1011-1057.

\bibitem{Lin1} Y.-C. Lin, H. Wang and K.-C. Wu, Explicit structure of the Fokker-Planck equation with potential. Quart. Appl. Math., 77 (2019), 727-766.

\bibitem{Lin2}  Y.-C. Lin, H. Wang and K.-C. Wu, Quantitative Pointwise Estimate of the Solution of the Linearized Boltzmann Equation, J. Stat. Phys., 171 (2018), 927-964.

\bibitem{Liu1} T.-P. Liu and S.-H. Yu, The Green's function and large-time behavior of solutions for the one-dimensional Boltzmann equation, {\it Comm. Pure Appl. Math.}, 57 (2004), 1543-1608.

\bibitem{Liu2} T.-P. Liu and S.-H. Yu, The Green's function of Boltzmann equation, 3D waves. {\it Bull. Inst.
Math. Acad. Sin. (N. S.)}, 1(1) (2006), 1-78.

\bibitem{Liu3} T.-P. Liu and S.-H. Yu, Solving Boltzmann equation, Part I: Green's function, Bull. Inst. Math. Acad. Sin. (N.S.), 6 (2011), 151-243.


\bibitem{Yu} L. Luo and H.-J. Yu, Spectrum analysis of the linear Fokker-Planck equation. Anal. Appl., 15(3) (2017), 313-331.

\bibitem{Markowich} P.A. Markowich, C.A. Ringhofer and C. Schmeiser, Semiconductor Equations, Springer-Verlag, Vienna, 1990. x+248 pp.



\bibitem{FL-4} J. Nieto, F. Poupaud and J. Soler, High-field limit for the Vlasov-Poisson-Fokker-Planck system. Arch. Ration. Mech. Anal., 158(1), (2001), 29-59.


\bibitem{FL-5} F. Poupaud and J. Soler,  Parabolic limit and stability of the Vlasov-Fokker-Planck
system. Math. Models Methods Appl. Sci., 10(7), (2000), 1027-1045.

\bibitem{VPFP-2} G. Rein and J. Weckler, Generic global classical solutions of the Vlasov-Fokker-Planck-Poisson system in three dimensions. J. Differ. Equ., 99(1), (1992), 59-77.



\bibitem{Tanski1} I. A. Tanski, Fundamental solution of Fokker-Planck equation. arXiv: nlin/0407007, 2004.

\bibitem{Tanski2} I. A. Tanski, Fundamental solution of degenerated Fokker-Planck equation. arXiv: 0804.0303, 2008.

\bibitem{FL-6}H. Wu, T.-C. Lin and C. Liu, Diffusion Limit of Kinetic Equations
for Multiple Species Charged Particles, Arch. Rational Mech. Anal., 215 (2015), 419-441.

\bibitem{VPFP-7} H.D. Victory, On the existence of global weak solutions for Vlasov-Poisson-Fokker-Planck systems. J. Math. Anal. Appl., 160(2), (1991), 525-555.

\bibitem{VPFP-3} H.D. Victory and B.P. O'Dwyer, On classical solutions of Vlasov-Poisson-Fokker-Planck systems. Indiana Univ. Math. J., 39(1), (1990), 105-156.




\end{thebibliography}
\end{document}